\RequirePackage{fix-cm}
\RequirePackage{amsmath}
\PassOptionsToPackage{numbers,sort&compress,square}{natbib}
\PassOptionsToPackage{style=base}{caption}
\documentclass[3p, authoryear]{elsarticle}
\ifdefined \pdfoutput
  \pdfoutput=1
\fi
\usepackage{bibfix}

\usepackage[T1]{fontenc}
\usepackage[english]{babel}
\usepackage{newtxtext}
\usepackage{newtxmath}
\usepackage{mathrsfs}

\usepackage[colorlinks,linkcolor=blue,citecolor=blue,urlcolor=blue]{hyperref}
\usepackage{doi}

\usepackage{xcolor}
\usepackage{bm}
\usepackage{microtype}
\usepackage{mathtools}
\usepackage{leftidx}
\usepackage{ifluatex}
\usepackage{enumitem}
\usepackage{graphicx}
\usepackage{booktabs}

\usepackage{siunitx}

\usepackage{support-caption}
\usepackage{subcaption}

\usepackage[draft,layout={inline,index}]{fixme}
\fxsetup{theme=colorsig}
\FXRegisterAuthor{jek}{ajek}{JEK}
\FXRegisterAuthor{lcw}{alcw}{LCW}
\FXRegisterAuthor{fxg}{afxg}{FXG}
\FXRegisterAuthor{mw}{amw}{MW}
\FXRegisterAuthor{thg}{athg}{THG}

\usepackage{pgfplots}
\pgfplotsset{compat=1.13}

\makeatletter
\newcommand{\footnoteref}[1]{%
\ltx@ifpackageloaded{hyperref}{% hyperref package loaded
  \ifHy@hyperfootnotes% option hyperfootnotes=true
    \hbox{\hyperref[#1]{%
            \@textsuperscript {\normalfont \ref*{#1}}}}%
  \else% option hyperfootnotes=false
    \hbox{\@textsuperscript {\normalfont \ref*{#1}}}%
  \fi%
}{% hyperref package not loaded
    \hbox{\@textsuperscript {\normalfont \ref{#1}}}%
 }%
}
\makeatother
\let\originalleft\left
\let\originalright\right
\renewcommand{\left}{\mathopen{}\mathclose\bgroup\originalleft}
\renewcommand{\right}{\aftergroup\egroup\originalright}

\usepackage{lineno}
\newcommand*\patchAmsMathEnvironmentForLineno[1]{%
  \expandafter\let\csname old#1\expandafter\endcsname\csname #1\endcsname
  \expandafter\let\csname oldend#1\expandafter\endcsname\csname end#1\endcsname
  \renewenvironment{#1}%
     {\linenomath\csname old#1\endcsname}%
     {\csname oldend#1\endcsname\endlinenomath}}%
\newcommand*\patchBothAmsMathEnvironmentsForLineno[1]{%
  \patchAmsMathEnvironmentForLineno{#1}%
  \patchAmsMathEnvironmentForLineno{#1*}}%
\AtBeginDocument{%authoryear
\patchBothAmsMathEnvironmentsForLineno{equation}%
\patchBothAmsMathEnvironmentsForLineno{align}%
\patchBothAmsMathEnvironmentsForLineno{flalign}%
\patchBothAmsMathEnvironmentsForLineno{alignat}%
\patchBothAmsMathEnvironmentsForLineno{gather}%
\patchBothAmsMathEnvironmentsForLineno{multline}%
}
\usepackage[autostyle]{csquotes}

\ifluatex
\providecommand{\versionortoday}{%
  \directlua{
    local f = io.popen('git describe --dirty')
    if f == nil then
    else
      tex.print('Version: ', f:read())
      f:close()
    end
  }
}
\else
\providecommand{\versionortoday}{\today}
\fi

\newcommand\reva[1]{{#1}}
\newcommand\revb[1]{{#1}}
\newcommand\revs[1]{{#1}}
\newcommand\revo[1]{{#1}}

\newcommand{\pnt}[1] {\ensuremath{\bm{#1}}}
\renewcommand{\vec}[1] {\ensuremath{\bm{#1}}}
\newcommand{\mat}[1] {\ensuremath{\bm{#1}}}

\renewcommand{\Re}{{\mathbb{R}}}
\newcommand{\PS}{{\mathbb{K}}}
\newcommand{\refelm}[1] {\ensuremath{\hat{#1}}}

\newcommand{\der}  [2]{\frac{d#1}{d#2}}
\newcommand{\pder} [2]{\frac{\partial#1}{\partial#2}}

\newcommand{\dof}[2]{{\left[#1\right]}_{#2}}
\newcommand{\Fec}{{\vec{\mathcal{F}}}}
\newcommand{\Hec}{{\vec{\mathcal{H}}}}
\newcommand{\Dec}{{\vec{\mathcal{D}}}}
\newcommand{\Aec}{{\mat{\mathcal{A}}}}
\newcommand{\avg}[1]{\ensuremath{\left\{\!\left\{#1\right\}\!\right\}}}
\newcommand{\jmp}[1]{\ensuremath{\left[\!\left[#1\right]\!\right]}}
\newcommand{\basis}{{\reva{\upsilon}}}
\newcommand{\Basis}{{\reva{\Upsilon}}}
\newcommand{\volm}{{\;\reva{d\vec{x}}}}
\newcommand{\refvolm}{\;\reva{d\refelm{\vec{x}}}}
\newcommand{\surfm}{\;\reva{dS}}
\newcommand{\refsurfm}{\;\reva{d\refelm{S}}}
\DeclareMathOperator{\Diagonal}{\revs{diag}}

\newtheorem{theorem}{Theorem}
\newtheorem{lemma}[theorem]{Lemma}

\newdefinition{remark}{Remark}
\newdefinition{assumption}{Assumption}
\newproof{proof}{Proof}

\begin{document}

\begin{frontmatter}

\title{Entropy Stable Discontinuous Galerkin Methods for Balance Laws in
  Non-Conservative Form: Applications to \revo{the Euler Equations with Gravity}}

\author{Maciej Waruszewski\fnref{nps}}
\ead{mwarusz@igf.fuw.edu.pl}
\author{Jeremy~E. Kozdon\fnref{nps}}
\ead{jekozdon@nps.edu}
\author{Lucas~C. Wilcox\fnref{nps}}
\ead{lwilcox@nps.edu}
\author{Thomas~H. Gibson\fnref{ncsa}}
\ead{thgibso2@illinois.edu}
\author{Francis~X. Giraldo\fnref{nps}}
\ead{fxgirald@nps.edu}

\date{\versionortoday}
\fntext[nps]{
  Department of Applied Mathematics,
  Naval Postgraduate School, Monterey, CA, USA,
}
\fntext[ncsa]{
  National Center for Supercomputing Applications,
  Urbana, IL, USA,
}
  \nonumnote{The views expressed in this document are those of the
  authors and do not reflect the official policy or position of the Department
  of Defense or the U.S.  Government.\\ Approved for public release;
  distribution unlimited\\}

\begin{abstract}
  In this work a non-conservative balance law formulation is considered that
  encompasses the rotating, compressible Euler equations for dry atmospheric
  flows.
  We develop a semi-discretely entropy stable discontinuous Galerkin method on
  curvilinear meshes using a generalization of flux differencing for numerical
  fluxes in fluctuation form.
  The method uses the skew-hybridized formulation of the element operators to
  ensure that, even in the presence of under-integration on curvilinear meshes,
  the resulting discretization is entropy stable.
  Several atmospheric flow test cases in one, two, and three dimensions
  confirm the theoretical entropy stability results as well as show the
  high-order accuracy and robustness of the method.
\end{abstract}


\end{frontmatter}

\section{Introduction}

Discontinuous Galerkin (DG) schemes for hyperbolic equations
have many favorable properties, such as high-order accuracy, compact
stencils, geometric flexibility, and the ability to add physical dissipation through
upwind-biased numerical fluxes.
DG belongs to the class of high-order element-based schemes \cite{Hesthaven2002,kopriva:2009,giraldo:2020} that
can harness the fine-grained GPU parallelism
needed for efficient exascale discretizations \cite{Abdi2019, Kolev2021, Abdelfattah2021}.
Due to their compact nature, DG schemes have demonstrated excellent
scalability \cite{Muller2019}.
Provably stable DG discretizations can be constructed for linear
systems \cite{Warburton2010, Kopriva2014}, and
often work well for nonlinear systems with smooth and reasonably well-resolved
solutions.
However, in the presence of discontinuities and under-resolved phenomena,
DG schemes often  require additional stabilization.
Classical stabilization techniques include
dealiasing through over-integration \cite{Mengaldo2015},
spectral filters \cite{Fischer2001}, limiters \cite{Qiu2005, Krivodonova2007},
and artificial viscosity \cite{Ullrich2018, Yu2020}.
While these techniques enabled DG to perform complex
simulations~\cite{Beck2014}, they are not without drawbacks such as
the possible loss of accuracy, hand-tuning of
multiple parameters, and lack of a solid theoretical foundation.

Recently, entropy-stable DG schemes
have emerged as a way to construct
high-order discretizations with rigorous nonlinear stability estimates
in a parameter-free fashion \cite{Gassner2021, Rojas2021}.
Namely, semi-discrete entropy stability estimates are guaranteed by using
specially constructed numerical fluxes~\cite{Tadmor1987} and
flux-differencing~\cite{Fisher2013}.
Importantly, the schemes are computationally affordable because they do not rely
on exact integration but instead exploit the
summation-by-parts property of polynomial differentiation operators~\cite{Gassner2013}.
Following the pioneering work in \citet{Carpenter2014, Gassner2016}
for the compressible Euler equations, the class of entropy-stable DG
schemes has been extended
to general systems of nonlinear conservation laws by \citet{Chen2017}.
\citet{Chan2018, Chan2019} generalized this to
``modal'' DG formulations which do not satisfy
summation-by-parts.
Most entropy stable DG schemes are constructed assuming exact integration in
time, and \citet{Ranocha2020} achieved fully discrete entropy-stability using
relaxation Runge--Kutta time integrators.

Entropy-stable DG schemes have been constructed for many systems
of nonlinear conservation and balance laws, such as
the shallow water equations \cite{Wintermeyer2017},
the compressible multi-component Euler equations \cite{Renac2021},
special relativistic ideal magneto-hydrodynamics \cite{Duan2020},
and the ten-moment Gaussian closure equations \cite{Biswas2021}.
However, many equations of physical interest do not conform to the
standard balance law form and require extensions of the basic
entropy-stability framework.
Entropy-stable treatment of non-conservative source terms
appearing in certain formulations of magneto-hydrodynamics
has been presented in \citet{Liu2018} for the ideal equations
and in \citet{Bohm2020} for the resistive equations.
\citet{Renac2019} constructed an entropy-stable DG spectral element scheme for
a one-dimensional, nonlinear
hyperbolic system in non-conservative form and applied it to two-phase flow.
\citet{Coquel2021} then used this framework for the Baer--Nunziato two-phase
flow model.

The motivation for this paper is the construction of entropy stable schemes for
the Euler equations with gravity. These equations are of fundamental importance
in simulations of astrophysical phenomena,
atmospheric large-eddy simulations,
numerical weather prediction, and climate modeling.
Due to their excellent properties, DG schemes are already emerging
in computational astrophysics \cite{Schaal2015, Kidder2017} as well as
in small- and large-scale atmospheric modeling
\cite{Giraldo2008, Marras2016, Tumolo2015, Baldauf2021}.
In contrast to engineering applications,
many atmospheric models formulate the equations in term\revo{s} of
an entropy-type variable known as potential temperature,
which has favorable properties for simulating atmospheric flows \cite{Smolarkiewicz2009}.
However, in climate modeling applications,
conservation of total energy and balanced energy transfers are also of great interest.
This can be achieved, for example, by using mimetic methods \cite{Taylor2020}.
An alternative, that we propose in this paper, is to guarantee conservation
of energy by using the total energy formulation, and
use the entropy-stable approach for robustness and consistency
with the second law of thermodynamics.
Crucially, this requires an extension of the existing entropy-stable framework,
since in the total energy formulation of the Euler equations
the gravity force enters as a non-conservative and spatially varying source
term.

To do this we recast the Euler equations with gravity in a more general,
non-conservative balance law formulation. We then
construct arbitrary order flux differencing DG schemes for this balance law
formulation and give conditions that the numerical fluxes must satisfy to ensure
entropy stability.
We present, in explicit form, an entropy-conserving numerical flux for the
Euler equations with gravity and show how to modify it to make it entropy-stable.
A variation on the classical one-dimensional Sod shock tube benchmark, put into a gravitational field,
is used to show the versatility of the scheme  in terms of choice of basis and quadrature rule.
Two-dimensional atmospheric test cases of buoyant convection  and gravity wave propagation
verify, respectively, the entropy-conservation and high-order accuracy of the scheme.
Finally, a baroclinic wave benchmark on the sphere
demonstrates the robustness of the scheme for long-time simulations of idealized weather-like phenomena
in a curvilinear three-dimensional geometry.

\section{Balance law formulation}
The rotating, compressible Euler equations for dry atmospheric flows in $d$
dimensions ($d=1,2,3$) are
\begin{subequations}
  \label{eq:gov:atmos}
  \begin{align}
    &\pder{\rho}{t} + \sum_{k=1}^{d} \pder{\rho u_{k}}{x_{k}} = 0,\\
    &\pder{\rho u_{i}}{t} + \sum_{k=1}^{d} \pder{\rho u_{i} u_{k}}{x_{k}}
    + \pder{p}{x_{i}}
    + \rho \pder{\phi}{x_{i}}
    =
    -2 \sum_{j,k=1}^{d} \varepsilon_{ijk} \omega_{j} u_{k},
    ~ i = 1,\dots, d,\\
    &\pder{\rho e}{t} + \sum_{k=1}^{d} \pder{u_{k} (\rho e + p)}{x_{k}} = 0.
  \end{align}
\end{subequations}
Here the prognostic variables are density $\rho$, momentum $\rho u_{i}$ for $i =
1, \dots, d$, and total energy $\rho e$. The pressure $p$ is
\begin{equation}
  p = (\gamma-1) \left(\rho e - \rho \phi - \frac{1}{2}\sum_{j=1}^{d}\rho u_{j}^{2}\right),
\end{equation}
where $\phi$ is a spatially varying geopotential function and the constant
$\gamma = \frac{c_p}{c_v}$ is the specific heat ratio given in terms of the
specific heats for constant pressure and volume, $c_p$ and $c_v$, respectively;
we take $\gamma = 7 / 5$. The components of a (potentially spatially varying) planetary
rotation vector are $\omega_{i}$ for $i = 1, \dots, d$. The Coriolis source
term is given as twice the cross product of the rotation vector with the fluid
velocity, expressed using the Levi--Civita symbol
\begin{equation}
  \label{eq:epsilon}
  \varepsilon_{ijk} =
  \begin{cases}
    +1, & \text{\ if\ } ijk \text{\ is\ } 123, 312, \text{or}\ 231,\\
    -1, & \text{\ if\ } ijk \text{\ is\ } 321, 132, \text{or}\ 213,\\
    \phantom{-}0, & \mbox{\ otherwise}.
  \end{cases}
\end{equation}
\revb{For the problem defined by equations~\eqref{eq:gov:atmos} to be
well-posed, hyperbolic, and physically meaningful, density and pressure must be positive.}

\reva{Designing an entropy stable discretization of governing equations~\eqref{eq:gov:atmos} 
is challenging because of the presence of the geopotential in the gravitational source and the pressure gradient term.
It is possible to rewrite~\eqref{eq:gov:atmos} as a pure conservation law~\cite{Chertock2018},
but it necessitates introduction of global fluxes, which are problematic for Godunov methods.
Here, we consider a generalized balance law
formulation of the form}
\begin{equation}
  \label{eq:gov}
  \pder{\vec{q}(\pnt{x}, t)}{t} +
  \sum_{i=1}^{d} \mat{A}_{i}\left(\vec{q}(\pnt{x}, t), \pnt{x}\right)
  \pder{\vec{h}_{i}\left(\vec{q}(\pnt{x}, t), \pnt{x}\right)}{x_{i}}
  =
  \vec{g}\left(\vec{q}(\pnt{x}, t), \pnt{x}\right),
\end{equation}
\reva{which will facilitate designing an entropy stable method.}
Here, $\vec{q}(\pnt{x}, t)$ is the solution vector of length $N_c$, the total
number of prognostic variables. Additionally, $\vec{h}_{i}\left(\vec{q},
\pnt{x}\right) \in \Re^{N_{d}}$ is a vector-valued function and
$\mat{A}_{i}\left(\vec{q}, \pnt{x}\right) \in \Re^{N_{c} \times N_{d}}$ is a
matrix-valued function, where in order to allow modeling flexibility $N_{d}$ is
allowed to be different from $N_{c}$; for conservation laws $N_{d} = N_{c}$ and
$\mat{A}_{i} = I$ with $\vec{h}_{i}$ being the flux function.  Balance law
formulation~\eqref{eq:gov} allows for source terms to be included in
$\mat{A}_{i}\pder{\vec{h}_{i}}{x_{i}}$ or the
vector-valued function $\vec{g}\left(\vec{q}, \pnt{x}\right) \in \Re^{N_{c}}$.

With $d = 3$ the atmospheric equations~\eqref{eq:gov:atmos} can be written in
balance law form~\eqref{eq:gov} by defining the solution vector
\begin{equation}
  \vec{q} = \begin{bmatrix}
    \rho &
    \rho u_{1} &
    \rho u_{2} &
    \rho u_{3} &
    \rho e
  \end{bmatrix}^{T}.
\end{equation}
The matrix function $\mat{A}_{k}$ and vector function $\vec{h}_{k}$ for $k = 1,
2, 3$ have $N_{d} = 6$ and are defined as
\begin{equation}
  \begin{split}
    \mat{A}_{k}(\vec{q}, \vec{x})
    =
    \begin{bmatrix}
      1 & 0 & 0 & 0 & 0 & 0\\
      0 & 1 & 0 & 0 & 0 & \delta_{1k} \rho\\
      0 & 0 & 1 & 0 & 0 & \delta_{2k} \rho\\
      0 & 0 & 0 & 1 & 0 & \delta_{3k} \rho\\
      0 & 0 & 0 & 0 & 1 & 0\\
    \end{bmatrix},
    \quad
    \vec{h}_{k}(\vec{q}, \vec{x})
    =
    \begin{bmatrix}
      \rho u_{k}\\
      \rho u_{1}u_{k} + \delta_{1k} p\\
      \rho u_{2}u_{k} + \delta_{2k} p\\
      \rho u_{3}u_{k} + \delta_{3k} p\\
      u_{k} \left(\rho e + p\right)\\
      \phi
    \end{bmatrix},
  \end{split}
\end{equation}
with $\delta_{ik}$ being the Kroneckor delta.
The source term is
\begin{equation}
  \label{eq:source}
  \vec{g}\left(\vec{q}, \pnt{x}\right) =
  \sum_{j,k=1}^{d}
  \begin{bmatrix}
    0 \\
    -2\varepsilon_{1jk}\omega_{j}u_{k} \\
    -2 \varepsilon_{2jk}\omega_{j}u_{k} \\
    -2 \varepsilon_{3jk}\omega_{j}u_{k} \\
    0
  \end{bmatrix};
\end{equation}
note that only the Coriolis term is included in the source and the geopotential
is included in the differential term.

\section{Entropy analysis}
Here we summarize continuous entropy results needed for the discrete analysis
that follows; for a more complete overview of the theory see, for example,
\citet[~Sections~1.4, 1.5, 3.1, and~3.2]{Dafermos2010}.

\subsection{Companion balance law}
We assume that balance law~\eqref{eq:gov} is endowed with a
nontrivial companion balance law
\begin{equation}
  \label{eq:companion_balance_law}
  \pder{\eta(\vec{q}(\pnt{x}, t), \pnt{x})}{t} +
  \sum_{j=1}^{d} \pder{\zeta_{j}(\vec{q}(\pnt{x}, t), \pnt{x})} {x_{j}} =
  \Pi(\vec{q}(\pnt{x}, t), \pnt{x}),
\end{equation}
where $\eta$ is a scalar convex entropy function and $\zeta_{j}$ are
associated entropy fluxes. This assumption is valid for many systems of
balance laws from continuum physics which typically have a natural companion
balance law related to the second law of thermodynamics, e.g.,
\citet[~Section~3.3]{Dafermos2010}. Specifically, companion balance
law~\eqref{eq:companion_balance_law} provides an entropy balance constraint on
the solution of balance law~\eqref{eq:gov}.  The vector of
entropy variables $\vec{\beta}$ is obtained by differentiating the entropy
function with respect to the state $\vec{q}$. That is, $\vec{\beta} =
\frac{\partial \eta}{\partial \vec{q}}$, with component $\alpha = 1, \dots,
N_{c}$ given by
\begin{equation}
  \label{eq:entropy_variables_def}
  \beta_{\alpha}(\vec{q}, \pnt{x})
  = \frac{\partial \eta(\vec{q}, \pnt{x})}{\partial q_{\alpha}}.
\end{equation}

To obtain an entropy balance condition, we first contract the
balance law~\eqref{eq:gov} with the entropy
variables~\eqref{eq:entropy_variables_def},
\begin{equation}
  \label{beta_dot_gov}
  \pder{\eta(\vec{q}, \pnt{x})}{t} +
  \sum_{i=1}^d
  \vec{\beta}^{T}(\vec{q}, \pnt{x})
  \mat{A}_{i}(\vec{q}, \pnt{x})
  \pder{\vec{h}_{i}(\vec{q}, \pnt{x})}{x_{i}}
  = \vec{\beta}^T(\vec{q}, \pnt{x}) \vec{g}(\vec{q}, \pnt{x}).
\end{equation}
Comparing~\eqref{beta_dot_gov} with the companion balance
law~\eqref{eq:companion_balance_law} gives
\begin{equation}
  \label{eq:entropy:flux}
  \pder{\zeta_{j}(\vec{q}, \pnt{x})}{q_{\alpha}}
    =
    \vec{\beta}^{T}(\vec{q}, \pnt{x})
    A_{j}(\vec{q}, \pnt{x})\pder{\vec{h}_{j}(\vec{q}, \pnt{x})}{q_{\alpha}},
\end{equation}
when the entropy production function is taken as
\begin{equation}
  \label{eq:entropy_production}
  \Pi(\vec{q}, \pnt{x}) =
  {\vec{\beta}(\vec{q}, \pnt{x})}^{T} \vec{g}(\vec{q}, \pnt{x})
  + \sum_{i = 1}^{d} \left(
    \pder{\zeta_{i}(\vec{q}, \pnt{x})}{x_{i}}
    -
    \vec{\beta}^{T}(\vec{q}, \pnt{x})
    \mat{A}_{i}(\vec{q}, \pnt{x})
    \pder{\vec{h}_{i}(\vec{q}, \pnt{x})}{x_{i}}
  \right).
\end{equation}
We will restrict ourselves to equations for which the summation term is exactly
zero so that $\Pi(\vec{q}, \pnt{x}) = {\vec{\beta}(\vec{q}, \pnt{x})}^{T}
\vec{g}(\vec{q}, \pnt{x})$.

\revb{Classical (Lipschitz continuous) solutions that satisfy the balance
law~\eqref{eq:gov} (almost everywhere) are automatically (classical)
solutions of the companion balance law~\eqref{eq:companion_balance_law}.}
Integrating~\eqref{eq:companion_balance_law} over the
domain $\Omega$ with boundary $\partial \Omega$ and outward unit normal
$\vec{n}$ gives the continuous entropy balance equality
\begin{equation}
  \label{eq:entropy:integral}
  \int_{\Omega} \pder{\eta(\vec{q}, \pnt{x})}{t} \volm
  +
  \sum_{j=1}^{d}
  \int_{\partial \Omega}
  n_{j} \zeta_{j}(\vec{q}, \pnt{x}) \surfm
  =
  \int_{\Omega} \Pi(\vec{q}, \pnt{x}) \volm.
\end{equation}
\revb{\citet[Section~1.5]{Dafermos2010} further states that one of the tenets of
the theory of balance laws is that admissible weak solutions should at least
satisfy
\begin{equation}
  \pder{\eta(\vec{q}(\pnt{x}, t), \pnt{x})}{t} +
  \sum_{j=1}^{d} \pder{\zeta_{j}(\vec{q}(\pnt{x}, t), \pnt{x})}{x_{j}}
  \le \Pi(\vec{q}(\pnt{x}, t), \pnt{x})
\label{eq:companion_balance_law_admissible}
\end{equation}
in the sense of distributions. Integrating this over the domain gives:}
\begin{equation}
  \label{eq:entropy:integral-inequality}
  \int_{\Omega} \pder{\eta(\vec{q}, \pnt{x})}{t} \volm
  +
  \sum_{j=1}^{d}
  \int_{\partial \Omega}
  n_{j} \zeta_{j}(\vec{q}, \pnt{x}) \surfm
  \leq
  \int_{\Omega} \Pi(\vec{q}, \pnt{x}) \volm.
\end{equation}
The primary objective of semi-discrete entropy analysis is to ensure
a discrete analogue of entropy balance~\eqref{eq:entropy:integral}
and dissipation~\eqref{eq:entropy:integral-inequality}.

\subsection{Entropy for the atmospheric Euler equations}
For the atmospheric Euler equations~\eqref{eq:gov:atmos}, we define the convex
mathematical entropy function
\begin{equation}
 \label{eq:entropy:variable}
  \eta =  -\frac{\rho s}{\gamma - 1},
\end{equation}
where $s = \log(p / \rho^\gamma)$ is the specific (physical) entropy.
The components of the entropy flux~\eqref{eq:entropy:flux} are taken to be
\begin{equation}
  \zeta_{j} = u_{j} \eta.
\end{equation}
The entropy variables~\eqref{eq:entropy_variables_def} are 
obtained with $\eta$ defined by~\eqref{eq:entropy:variable} resulting in
\begin{subequations}
  \label{eq:entropy_variables_euler}
  \begin{align}
    \beta_{1} &= \frac{\gamma - s}{\gamma - 1} - \left(u_{1}^2 + u_{2}^2 + u_{3}^2 - 2\phi\right) b,\\
    \beta_{2} &= 2b u_{1},\\
    \beta_{3} &= 2b u_{2},\\
    \beta_{4} &= 2b u_{\revo{3}},\\
    \beta_{5} &= -2b,
  \end{align}
\end{subequations}
with $b = \rho / 2p$ is the inverse temperature. Recognizing that
$b = -\beta_{5} / 2$, the prognostic state can be recovered from the entropy
variables using
\begin{subequations}
  \begin{align}
    u_{1} &= \frac{\beta_{2}}{2b},\\
    u_{2} &= \frac{\beta_{3}}{2b},\\
    u_{3} &= \frac{\beta_{4}}{2b},\\
    e &= \frac{1}{(\gamma-1) 2b} + \frac{u_{1}^2 + u_{2}^2 + u_{3}^2}{2} + \phi,\\
    \rho &= {
      \left(
        2b
        \exp
        \left(
          \left(\gamma - 1\right)
          \left(-\beta_{1} + \left(2\phi - \left(u_{1}^2 + u_{2}^2 +
                u_{3}^2\right)\right) b\right) +
          \gamma
        \right)
      \right)
    }^{-1/\left(\gamma - 1\right)}.
  \end{align}
\end{subequations}
\revb{When pressure and density are both positive the mapping between entropy and conserved variables is well-defined.}
\revs{
Using explicit forms of the source term~\eqref{eq:source} and
the entropy variables~\eqref{eq:entropy_variables_euler} 
the entropy production for the
atmospheric Euler equations~\eqref{eq:gov:atmos} is
 $\Pi = \vec{\beta}^T \vec{g} = -2 b \epsilon_{ijk} \omega_j u_i u_k = 0$, which is zero
due to skew symmetry of $\epsilon$~\eqref{eq:epsilon}.
This result also simply follows from the fact that the Corolis force is an inertial force.
}

\section{Flux differencing discontinuous Galerkin methods}

In this section we establish our notation and a flux differencing
formulation for balance law~\eqref{eq:gov};  our notation
is motivated by \citet{Chan2018, Chan2019, Renac2019, Castro2013}.

\subsection{Notation}
Let the domain $\Omega$ be tessellated by a mesh of disjoint elements
$K \subset \Re^d$ with boundaries $\partial K$, where each polygonal element $K$
is mapped to a
reference element $\refelm{K}$ with boundary $\partial\refelm{K}$.
Points on the physical and reference elements are denoted as $\pnt{x}$
and $\pnt{\xi}$, respectively. The outward unit normal in the physical
coordinate system is $\vec{n}$ with $\refelm{\vec{n}}$ denoting the unit normal
to $\partial\refelm{K}$.

For each element $K$, we assume there is a suitably smooth (differentiable)
transformation between the physical and reference element. For a given function
$u(\pnt{x})$ defined on a physical element $K$, we denote its mapping to the
reference element $\refelm{K}$ as $\refelm{u}(\pnt{\xi})$.
We let $J^{v}(\pnt{\xi})$ be the Jacobian determinant of the
volume transformation\revo{,} $J^{f}(\pnt{\xi})$ the Jacobian determinant of the
surface mapping\revo{, and $\refelm{J}^{f}(\pnt{\xi})$ the Jacobian determinant of the
associated reference face}. Using these Jacobian factors, integrals over
physical elements in terms of integrals on the reference element are
\begin{align}
  \int_{K} u(\pnt{x}) \volm &= \int_{\refelm{K}} J^{v}(\pnt{\xi}) \refelm{u}(\pnt{\xi}) \refvolm,\\
  \int_{\partial K} u(\pnt{x}) \surfm &= \int_{\partial \refelm{K}} J^{f}(\pnt{\xi})
  \refelm{u}(\pnt{\xi}) \refsurfm.
\end{align}
The reference and physical outward unit normals are related by
\begin{align}
  J^{f} n_{i} = \sum_{k=1}^{d} J^{v} \pder{\xi_{i}}{x_{k}} \refelm{n}_{k}
  \revo{\refelm{J}^{f}_{k}}.
\end{align}

On the reference element we assume that functions are approximated in some
polynomial space $\PS^{N} = \PS^{N}(\refelm{K})$, where $N$ is the degree of
approximation. In the numerical results we use tensor product (interval,
quadrilateral, and hexahedral) elements to represent the computational domain,
and $\PS^{N}$ is selected to be the space of $d$-dimensional tensor product
polynomials of \revb{degree $N$ in each dimension}. However, the analysis that follows is more
general and can be applied to other elements types, e.g., simplices; see
\citet{Chan2019} for a fuller discussion of other element types.
Specifically, in the definition of the operators there is no assumption that the
interpolation and quadrature points are collocated. Additionally, we do not
assume that the surface quadrature points are a subset of the volume quadrature
points. That said, when Legendre--Gauss--Lobatto points are used for
interpolation and quadrature, e.g., a spectral element method, many of the
operators simplify; see Remark \ref{remark:lgl} below.

We let ${\left\lbrace \basis_{i}\left(\pnt{\xi}\right) \right\rbrace}_{i=1}^{N_{p}}$
denote a basis for $\PS^{N}$, where $N_p$ is the dimensionality of the space.
Any function $q^{N}\left(\pnt{\xi}\right) \in \PS^{N}$ can therefore be
represented as
\begin{align}
  q^{N}\left(\pnt{\xi}\right)
  =
  \vec{\Basis}^{T}\left(\pnt{\xi}\right)
  \vec{q}^{N},
\end{align}
where $\mat{\Basis}\left(\pnt{\xi}\right)$ is a vector of basis functions
and $\vec{q}^{N}$ is a vector of expansion coefficients.

\subsection{Skew-hybridized summation-by-parts operators}

Following \citet{Chan2019}, we construct skew-hybridized operators involving
both volume and surface quadrature nodes. Here we highlight the notation needed in
this paper; further details and discussion on the operators' accuracy requirements
can be found in \citet{Chan2019}. Throughout we use the
notation $\dof{\cdot}{n}$ and $\dof{\cdot}{n,m}$ to denote the elements
of a vector and matrix, respectively.

The volume and surface quadrature rules on the reference
element $\refelm{K}$ are
$\left(\dof{\vec{\xi}^{v}}{i}, \dof{\vec{w}^{v}}{i}\right)_{i=1}^{N_{q}^{v}}$
and
$\left(\dof{\vec{\xi}^{f}}{i}, \dof{\vec{w}^{f}}{i}\right)_{i=1}^{N_{q}^{f}}$,
with $N_{q}^{v}$ and $N_{q}^{f}$ being the total number of volume and surface
quadrature points, respectively.
Here $\vec{\xi}^{v}$ and $\vec{\xi}^{f}$ denote vectors of quadrature points,
while $\vec{w}^{v}$ and $\vec{w}^{f}$ are the weights.
Diagonal matrices of the quadrature weights are
\begin{equation}
  \mat{W}^{v} = \Diagonal\left(\vec{w}^{v}\right),\quad
  \mat{W}^{f} = \Diagonal\left(\vec{w}^{f}\right).
\end{equation}
Similarly, we define diagonal matrices of the volume and surface Jacobians
evaluated at the quadrature points,
\begin{equation}
  \mat{J}^{v} = \Diagonal\left(J^{v}\left(\vec{\xi}^{v}\right)\right),\quad
  \mat{J}^{f} = \Diagonal\left(J^{f}\left(\vec{\xi}^{f}\right)\right).
\end{equation}
As discussed later, $J^{v}$ and $J^{f}$ are often not exact but approximations
in $\PS^{N}$.

Vandermonde matrices for the volume $\mat{V}^{v} \in \Re^{N_{q}^{v} \times N_{p}}$ and
surface $\mat{V}^f \in \Re^{N_{q}^{f} \times N_{p}}$ that interpolate the expansion
coefficients to the quadrature points have elements
\begin{equation}
  \dof{\mat{V}^{v}}{n,m} = \basis_m(\dof{\pnt{\xi}^{v}}{n}),
  \quad
  \dof{\mat{V}^{f}}{n,m} = \basis_m(\dof{\pnt{\xi}^{f}}{n}).
\end{equation}
The combined quadrature Vandermonde matrix
\begin{equation}
  \mat{V}
  =
  \begin{bmatrix}
    \mat{V}^{v}\\
    \mat{V}^{f}
  \end{bmatrix},
\end{equation}
interpolates from the expansion coefficients to the volume and
surface quadrature points.
With the above notation the volume mass matrix
\begin{equation}
  \mat{M}^{v} = {\left(\mat{V}^{v}\right)}^T\mat{W}^{v} \mat{J}^{v}\mat{V}^{v},
\end{equation}
is defined to approximate inner products of functions in $\PS^{N}$ on the
reference element in physical space.

In order to project functions defined on the volume quadrature points to
polynomials in $\PS^{N}$ we define a \emph{quadrature-based $L^2$ projection
operator}
\begin{equation}
  \label{eq:quad:proj}
  \mat{P}^{v} =
  {\left(\mat{M}^{v}\right)}^{-1}{\left(\mat{V}^{v}\right)}^T\mat{J}^{v}\mat{W}^{v}.
\end{equation}
We observe that going from the expansion coefficients to the quadrature points
and back is exact:
\begin{equation}\label{eq:pqvqinverse}
    \mat{P}^{v} \mat{V}^{v} = \mat{I}.
\end{equation}

The polynomial differentiation matrix with respect to the $i$-th reference
coordinate is $\mat{D}^{N}_i \in \Re^{N_p \times N_p}$.
Namely, if $\vec{q}$ are the expansion coefficients of
$q^{N} \in \PS^{N}$, then $\mat{D}_{i}^{N} \vec{q}^{N}$ are the expansion
coefficients for $\pder{q^{N}}{\xi_{i}} \in \PS^{N}$.  The
\emph{quadrature-based volume integrated differentiation operator} with respect
to the $i$-th reference coordinate is
\begin{equation}
  \mat{Q}^{v}_{i} = \mat{W}^{v} \mat{V}^{v} \mat{D}^{N}_{i} \mat{P}^{v}.
\end{equation}
With this, we follow \citet{Chan2019} and define the skew-hybridized
summation-by-parts (SBP) operator
\begin{equation}\label{eq:skewhybridSBPoperator}
  \mat{Q}_i =
    \frac{1}{2}\begin{bmatrix}
      \mat{Q}^{v}_{i} - {\left(\mat{Q}^{v}_{i}\right)}^T & {\left(\mat{E}^{v}\right)}^T\mat{B}^{f}_i \\
      -\mat{B}^{f}_i\mat{E}^{v} & \mat{B}^{f}_i
    \end{bmatrix}.
\end{equation}
Here the matrix $\mat{E}^{v} = \mat{V}^{f}\mat{P}^{v}$ projects between
the volume quadrature to the face quadrature and $\mat{B}^{f}_i = \mat{W}^f
\refelm{\mat{N}}^{f}_i \revo{\refelm{\mat{J}}^{f}}$ is the reference normal-weighted
surface quadrature with $\refelm{\mat{N}}^{f}_i$ being the diagonal matrix of
component $i$ of $\refelm{\vec{n}}$ \revo{and $\refelm{\mat{J}}^{f}$ being the
reference face Jacobian determinant, all} evaluated at the surface quadrature points.
Since $\mat{Q}_{i} \in \Re^{\left(N_{q}^{v} + N_{q}^{f}\right) \times
\left(N_{q}^{v} + N_{q}^{f}\right)}$ it operates on functions defined on
the \reva{combined volume and surface quadratures}.

The skew-hybridized SBP operator is almost skew-symmetric, as characterized by
the following lemma from \citet[Lemma 1]{Chan2019}.
\begin{lemma}[The skew-hybridized SBP operator]
  \label{lemma:properties-skewhybirdSBP}
  The skew-hybridized SBP operator
  $\mat{Q}_i$ satisfies the
  generalized summation-by-parts (SBP) property:
  \begin{equation}
    \label{eq:sbp:prop}
    \mat{Q}_i +
    \mat{Q}_i^T
    = \begin{bmatrix}
      \mat{0} \\
                & \mat{B}_i^{f}
    \end{bmatrix}.
  \end{equation}
\end{lemma}

The reason for $\mat{Q}_i$ being referred to as an SBP operator is
because the discrete identity
\begin{align}
  {\left(\vec{p}^{N}\right)}^{T}
  \mat{V}^{T}
  \left(
    \mat{Q}_i +
    \mat{Q}_i^{T}
  \right)
  \mat{V}
  \vec{q}^{N}
  =
  {\left(\mat{V}^{f} \vec{p}^{N}\right)}^{T}
  \mat{W}^{f}
  \refelm{\mat{N}}_i^{f}
  \mat{J}_i^{f}
  \mat{V}^{f}
  \vec{q}^{N},
\end{align}
is an approximation of the continuous \reva{multidimensional} integration-by-parts identity
\begin{align}
  \int_{\refelm{K}}
  \left(
    \pder{p}{\xi_{i}}
    q
    +
    p
    \pder{q}{\xi_{i}}
  \right)
  \refvolm
  =
  \int_{\partial\refelm{K}}
  \refelm{n}_{i}
  p
  q
  \refsurfm.
\end{align}
The advantage of using $\mat{Q}_i$ is that one is \emph{guaranteed} to satisfy
the SBP property, no matter \revb{the degree of accuracy} of the volume and surface quadrature
rules.  This allows for the construction of SBP operators in spite of aliasing
errors due to inexact integration
\citep{Gassner2013,Ranocha2018,Chan2019}.
Later we will assume consistency for the skew-hybridized
SBP operator; see \citet{Chan2019} for required accuracy conditions.
\begin{assumption}
  The skew-hybridized SBP operator is assumed to be exact for constants:
  \begin{equation}
    \label{eq:Qconst}
    \mat{Q}_{i} \vec{1} = \vec{0}.
  \end{equation}
\end{assumption}

\revb{
Diagonal matrices of the metric terms evaluated at volume and surface quadrature points are
\begin{equation}
  \mat{G}_{jk}^{v} = \Diagonal\left(J^{v}\pder{\xi_{j}}{x_{k}}\left(\vec{\xi}^{v}\right)\right),\quad
  \mat{G}_{jk}^{f} = \Diagonal\left(J^{v}\pder{\xi_{j}}{x_{k}}\left(\vec{\xi}^{f}\right)\right),
\end{equation}
respectively.  Their combined matrix is 
\begin{equation}
\mat{G}_{jk}
    = \begin{bmatrix}
      \mat{G}_{jk}^{v}  \\
                & \mat{G}_{jk}^{f} 
    \end{bmatrix}.
\end{equation}
}
As noted above, these are often approximated using the discrete operators.
\revs{One benefit of this is that the volume $\mat{G}_{jk}^{v}$ and surface
$\mat{G}_{jk}^{f}$ can be defined in a consistent manner, e.g.,
$\mat{G}_{jk}^{f} = \Diagonal\left(\mat{E}^{v}
\mat{\overline{G}}_{jk}^{v} 
\right)$ where $\mat{\overline{G}}_{jk}^{v}$ is the diagonal of $\mat{G}_{jk}^{v}$
interpreted as a vector.
}

In the discrete entropy stability analysis we will require that the metrics
discretely satisfy the geometric conservation law~\cite{Thomas1979,
Kopriva2006}
\begin{equation}
  \label{eq:gcl}
  \sum_{k=1}^{d} \mat{Q}_{k} \mat{G}_{jk} \vec{1} = \vec{0},
\end{equation}
which is a discrete statement of the continuous metric identity
\begin{equation}
  \sum_{k=1}^{d} \pder{}{\xi_{k}} \left(J \pder{\xi_{k}}{x_{j}}\right) = 0.
\end{equation}

\subsection{Flux differencing in fluctuation form}
Motivated by \citet{Castro2013} and \citet{Renac2019} we define a
numerical flux function in fluctuation form as
\begin{equation}
  \label{eq:num:fluctuation}
  \begin{split}
    \Dec_{i}\left(\vec{p}, \pnt{y}, \vec{q}, \pnt{x}\right)
      =\;&
      \Aec_{i}\left(\vec{p}, \pnt{y}, \vec{q}, \pnt{x}\right)
      \Hec_{i}\left(\vec{p}, \pnt{y}, \vec{q}, \pnt{x}\right)
      +
      \Aec_{i}\left(\vec{q}, \pnt{x}, \vec{p}, \pnt{y}\right)
      \Hec_{i}\left(\vec{q}, \pnt{x}, \vec{p}, \pnt{y}\right)\\
       &
       -
       \left(
         \Aec_{i}\left(\vec{p}, \pnt{y}, \vec{q}, \pnt{x}\right)
         +
         \Aec_{i}\left(\vec{q}, \pnt{x}, \vec{p}, \pnt{y}\right)
       \right)
       \vec{h}_{i}\left(\vec{p}, \pnt{y}\right)
       .
    \end{split}
  \end{equation}
The matrix-valued function
$\Aec_{i}\left(\vec{p}, \pnt{y}, \vec{q}, \pnt{x}\right)$
and vector-valued function
$\mat{\Hec}_{i}\left(\vec{p}, \pnt{y}, \vec{q}, \pnt{x}\right)$
are assumed to satisfy the consistency conditions
\begin{align}
  \label{eq:num:fluctuation:A}
  2\Aec_{i}\left(\vec{q}, \pnt{x}, \vec{q}, \pnt{x}\right) &= \mat{A}_{i}\left(\vec{q}, \pnt{x}\right),\\
  \label{eq:num:fluctuation:H}
  \Hec_{i}\left(\vec{q}, \pnt{x}, \vec{q}, \pnt{x}\right) &= \vec{h}_{i}\left(\vec{q}, \pnt{x}\right).
\end{align}
In the special case of balance law~\eqref{eq:gov} being a conservation
law,
i.e., $\mat{A}_{i} = I$ and $\vec{h}_{i} = \vec{f}_{i}$, we can choose
$\Aec_{i}\left(\vec{p}, \pnt{y}, \vec{q}, \pnt{x}\right) = I/2$ and let
$\Fec_{i}(\vec{p}, \pnt{y}, \vec{q}, \pnt{x})$ be a symmetric and consistent
numerical flux so that
\begin{equation}
  \label{eq:num:flux:cons}
  \Dec^{\text{cons}}_{i}\left(\vec{p}, \pnt{y}, \vec{q}, \pnt{x}\right)
  =
  \Fec_{i}\left(\vec{p}, \pnt{y}, \vec{q}, \pnt{x}\right)
  -
  \vec{f}_{i}\left(\vec{p}, \pnt{y}\right)
  .
\end{equation}

The following lemma relates the derivative of the
numerical flux function $\Dec_{i}$ to the spatial derivatives in balance
law~\eqref{eq:gov}, i.e., a flux differencing form of the derivative.
\begin{lemma}
  If numerical flux $\Dec_{i}$ defined by~\eqref{eq:num:fluctuation} satisifies
  consistency conditions~\eqref{eq:num:fluctuation:A}
  and~\eqref{eq:num:fluctuation:H} then
  \begin{equation}
    \label{eq:flux:diff}
    2\left.\pder{\Dec_{i}(\vec{q}(\pnt{y}), \pnt{y}, \vec{q}(\pnt{x}), \pnt{x})}{x_{j}}\right|_{\pnt{y}=\pnt{x}}
      =
      \mat{A}_{i}(\vec{q}(\pnt{x}), \pnt{x})
      \pder{\vec{h}_{i}(\vec{q}(\pnt{x}), \pnt{x})}{x_{j}}
    \end{equation}
  and
  \begin{equation}
    \label{eq:flux_diff_full}
    \begin{split}
      {\vec{p}(\pnt{x})}^{T} &\mat{A}_{j}(\vec{q}(\pnt{x}), \pnt{x})
      \pder{\vec{h}(\vec{q}(\pnt{x}), \pnt{x})}{x_{j}}\\
                             &=
                             {\left.
                                 \pder{
                                 }{x_{j}}
                                 \left(
                                   {\vec{p}(\pnt{y})}^{T}
                                   \Dec_{i}( \vec{q}(\pnt{y}), \pnt{y}, \vec{q}(\pnt{x}), \pnt{x})
                                   -
                                   {\vec{p}(\pnt{x})}^{T}
                                   \Dec_{i}( \vec{q}(\pnt{x}), \pnt{x}, \vec{q}(\pnt{y}), \pnt{y})
                                 \right)
                               \right|}_{\pnt{y}=\pnt{x}}.
    \end{split}
  \end{equation}
  \end{lemma}
\begin{proof}
  Identity~\eqref{eq:flux:diff} follows from the consistency conditions:
  \begin{equation}
    \begin{split}
      \mat{A}_{i}&(\vec{q}(\pnt{x}), \pnt{x})
      \pder{\vec{h}_{i}(\vec{q}(\pnt{x}), \pnt{x})}{x_{j}}\\
                 &=
                 \pder{\mat{A}_{i}(\vec{q}(\pnt{x}), \pnt{x}) \vec{h}_{i}(\vec{q}(\pnt{x}), \pnt{x})}{x_{j}}
                 -
                 \pder{\mat{A}_{i}(\vec{q}(\pnt{x}), \pnt{x})}{x_{j}} \vec{h}_{i}(\vec{q}(\pnt{x}), \pnt{x})\\
                 &=
                 {\left.\pder{}{x_j}
                     \left(
                       \mat{A}_{i}(\vec{q}(\pnt{x}), \pnt{x})
                       \left(\vec{h}_{i}(\vec{q}(\pnt{x}), \pnt{x})
                         -
                         \vec{h}_{i}(\vec{q}(\pnt{y}), \pnt{y})
                       \right)
                     \right)
                   \right|}_{\pnt{y} = \pnt{x}}\\
                 &=
                 {\left.2\pder{}{x_j}
                     \left(
                       \Aec_{i}(\vec{q}(\pnt{x}), \pnt{x}, \vec{q}(\pnt{x}), \pnt{x})
                       \left(\Hec_{i}(\vec{q}(\pnt{x}), \pnt{x}, \vec{q}(\pnt{x}), \pnt{x})
                         -
                         \vec{h}_{i}(\vec{q}(\pnt{y}), \pnt{y})
                       \right)
                     \right)
                   \right|}_{\pnt{y} = \pnt{x}}\\
                 &=
                 {\left.2\pder{}{x_j}
                     \left(
                       \Aec_{i}(\vec{q}(\pnt{y}), \pnt{y}, \vec{q}(\pnt{x}), \pnt{x})
                       \left(\Hec_{i}(\vec{q}(\pnt{y}), \pnt{y}, \vec{q}(\pnt{x}), \pnt{x})
                         -
                         \vec{h}_{i}(\vec{q}(\pnt{y}), \pnt{y})
                       \right)
                     \right)
                   \right|}_{\pnt{y} = \pnt{x}}\\
                 &\qquad+
                 {\left.2\pder{}{x_j}
                     \left(
                       \Aec_{i}(\vec{q}(\pnt{x}), \pnt{x}, \vec{q}(\pnt{y}), \pnt{y})
                       \left(\Hec_{i}(\vec{q}(\pnt{x}), \pnt{x}, \vec{q}(\pnt{y}), \pnt{y})
                         -
                         \vec{h}_{i}(\vec{q}(\pnt{y}), \pnt{y})
                       \right)
                     \right)
                   \right|}_{\pnt{y} = \pnt{x}}\\
                 &=
                 {\left.2\pder{}{x_j}
                     \mat{\Dec}_{i}(\vec{q}(\pnt{y}), \pnt{y}, \vec{q}(\pnt{x}), \pnt{x})
                   \right|}_{\pnt{y} = \pnt{x}}.
    \end{split}
  \end{equation}
  Consistency implies that $\Dec_{i}(\vec{q}, \pnt{x}, \vec{q}, \pnt{x}) =
  \vec{0}$, which when combined with~\eqref{eq:flux:diff}
  implies~\eqref{eq:flux_diff_full}:
  \begin{equation}
    \begin{split}
      {\vec{p}(\pnt{x})}^{T} &\mat{A}_{j}(\vec{q}(\pnt{x}), \pnt{x})
      \pder{\vec{h}(\vec{q}(\pnt{x}), \pnt{x})}{x_{j}}\\
                             &=
                             2{\vec{p}(\pnt{x})}^{T}
                             \left.
                               \pder{
                                 \Dec_{i}( \vec{q}(\pnt{y}), \pnt{y}, \vec{q}(\pnt{x}), \pnt{x})
                               }{x_{j}}\right|_{\pnt{y}=\pnt{x}}\\
                             &=
                             \left.
                               \left(
                             2{\vec{p}(\pnt{y})}^{T}
                               \pder{
                                 \Dec_{i}( \vec{q}(\pnt{y}), \pnt{y}, \vec{q}(\pnt{x}), \pnt{x})
                             }{x_{j}}\right)\right|_{\pnt{y}=\pnt{x}}
                               -
                               \pder{
                                 {\vec{p}(\pnt{x})}^{T}\Dec_{i}( \vec{q}(\pnt{x}), \pnt{x}, \vec{q}(\pnt{x}), \pnt{x})
                               }{x_{j}}\\
                             &=
                             {\left.
                                 \left(
                                   {\vec{p}(\pnt{y})}^{T}
                                   \pder{
                                     \Dec_{i}( \vec{q}(\pnt{y}), \pnt{y}, \vec{q}(\pnt{x}), \pnt{x})
                                   }{x_{j}}
                                   -
                                   \pder{
                                     {\vec{p}(\pnt{x})}^{T}
                                     \Dec_{i}( \vec{q}(\pnt{x}), \pnt{x}, \vec{q}(\pnt{y}), \pnt{y})
                                   }{x_{j}}
                                 \right)
                               \right|}_{\pnt{y}=\pnt{x}}\\
                             &=
                             {\left.
                                 \pder{
                                 }{x_{j}}
                                 \left(
                                   {\vec{p}(\pnt{y})}^{T}
                                   \Dec_{i}( \vec{q}(\pnt{y}), \pnt{y}, \vec{q}(\pnt{x}), \pnt{x})
                                   -
                                   {\vec{p}(\pnt{x})}^{T}
                                   \Dec_{i}( \vec{q}(\pnt{x}), \pnt{x}, \vec{q}(\pnt{y}), \pnt{y})
                                 \right)
                               \right|}_{\pnt{y}=\pnt{x}}.
    \end{split}
  \end{equation}
  \qed{}
\end{proof}

For conservation laws we have the following theorem showing the equivalence to
the conservation law flux differencing~\cite{Chan2018}.
\begin{theorem}
  If the numerical flux~\eqref{eq:num:flux:cons} is used then
  \begin{equation}
    \left.\pder{\Dec_{i}^{\text{cons}}(\vec{q}(\pnt{y}), \pnt{y}, \vec{q}(\pnt{x}), \pnt{x})}{x_{j}}\right|_{\pnt{y}=\pnt{x}}
      =
    \left.\pder{\Fec_{i}(\vec{q}(\pnt{y}), \pnt{y}, \vec{q}(\pnt{x}), \pnt{x})}{x_{j}}\right|_{\pnt{y}=\pnt{x}}.
  \end{equation}
\end{theorem}

In \citet{Renac2019} the case of $\vec{h}_{i}(\vec{q},\pnt{x}) = \vec{q}$ was
considered, and the above formulation is equivalent when
$\Hec_{i}^{\text{Renac}}(\vec{p}, \pnt{y}, \vec{q}, \pnt{x}) = \vec{q}$ in
which case the numerical flux reduces to
\begin{equation}
  \label{eq:num:flux:Renac}
  \Dec^{\text{Renac}}_{i}\left(\vec{p}, \pnt{y}, \vec{q}, \pnt{x}\right)
  =
  \Aec_{i}^{\text{Renac}}\left(\vec{p}, \pnt{y}, \vec{q}, \pnt{x}\right)
    \left(
      \vec{q}
      -
      \vec{p}
    \right)
    .
\end{equation}

\subsection{Discontinuous Galerkin methods with flux differencing in fluctuation form}

To derive a flux differencing scheme we begin with a ``strong derivative''
DG scheme for balance law~\eqref{eq:gov} for a single element $K$,
\begin{equation}
  \label{eq:init_var_form}
  \begin{split}
    \int_{K}
    &
    {\vec{p}(\pnt{x})}^{T}
    \pder{\vec{q}(\pnt{x},t)}{t}
    \volm
    +
    \sum_{j=1}^{d}\int_{K}
    {\vec{p}(\pnt{x})}^{T} \mat{A}_{j}(\vec{q}(\pnt{x}, t), \pnt{x})
    \pder{\vec{h}(\vec{q}(\pnt{x}, t), \pnt{x})}{x_{j}} \volm
    \\
    &
    =
    \int_{K}
    {\vec{p}(\pnt{x})}^{T}\vec{g}(\vec{q}(\pnt{x}, t), \pnt{x}) \volm
    -
    \sum_{j=1}^{d}
    \int_{\partial K}
    {\vec{p}(\pnt{x})}^{T}
    \Dec^{*}_{j}(\vec{q}^{-}(\pnt{x}, t), \pnt{x}^{-}, \vec{q}^{+}(\pnt{x}, t),
    \pnt{x}^{+})
    \;
    n_{j} \surfm.
  \end{split}
\end{equation}
Here $\vec{p}$ and $\vec{q}$ are the test and trial functions, and unless needed
for clarity no accent is added to denote that all functions are defined over
the element $K$.
The numerical flux on the face is $\Dec^{*}_{j}$ with the superscripts $-$ and
$+$ denoting values on the two-sides of the surface, with $-$ being the inside
value and $+$ the outside value; the unit normal $\vec{n}$ is outward to element
$K$.

Using the flux differencing identity~\eqref{eq:flux_diff_full}, the DG
scheme~\eqref{eq:init_var_form} can be rewritten as
\begin{equation}
  \label{eq:var_form}
  \begin{split}
    \int_{K}
    &
    {\vec{p}(\pnt{x})}^{T}
    \pder{\vec{q}(\pnt{x},t)}{t}
    \volm
    \\
    &
    +
    \sum_{j=1}^{d}\int_{K}
    {\left.
        \pder{
        }{x_{j}}
        \left(
          {\vec{p}(\pnt{y})}^{T}
          \Dec_{i}( \vec{q}(\pnt{y}, t), \pnt{y}, \vec{q}(\pnt{x}, t), \pnt{x})
          -
          {\vec{p}(\pnt{x})}^{T}
          \Dec_{i}( \vec{q}(\pnt{x}, t), \pnt{x}, \vec{q}(\pnt{y}, t), \pnt{y})
        \right)
      \right|}_{\pnt{y}=\pnt{x}} \volm
      \\
    &
    =
    \int_{K}
    {\vec{p}(\pnt{x})}^{T}\vec{g}(\vec{q}(\pnt{x}, t), \pnt{x}, t) \volm
    -
    \sum_{j=1}^{d}
    \int_{\partial K}
    {\vec{p}(\pnt{x})}^{T}
    \Dec^{*}_{j}(\vec{q}^{-}(\pnt{x}, t), \pnt{x}^{-}, \vec{q}^{+}(\pnt{x}, t),
    \pnt{x}^{+})
    \;
    n_{j} \surfm,
  \end{split}
\end{equation}
where the volume and surface numerical fluxes can differ.
Changing from the physical element $K$ to the reference element
$\refelm{K}$ gives the final variational form
\begin{equation}
  \label{eq:ref_var_form}
  \begin{split}
    \int_{\refelm{K}}
    &
    J^{v}(\pnt{\xi})
    {\vec{p}(\pnt{\xi})}^{T}
    \pder{\vec{q}(\pnt{\xi}, t)}{t} \refvolm
    \\
    &
    +
    \sum_{j,k=1}^{d}
    \int_{\refelm{K}}
    G_{jk}(\pnt{\xi})
    {\left.
        \pder{
        }{\xi_{k}}
        \left(
          {\vec{p}(\pnt{\eta})}^{T}
          \Dec_{i}( \vec{q}(\pnt{\eta}, t), \pnt{x}(\pnt{\eta}),
          \vec{q}(\pnt{\xi}, t), \pnt{x}(\pnt{\xi}))
        \right)
      \right|}_{\pnt{\eta}=\pnt{\xi}} \refvolm
      \\
    &
    -
    \sum_{j,k=1}^{d}
    \int_{\refelm{K}}
    G_{jk}(\pnt{\xi})
    {\left.
        \pder{
        }{\xi_{k}}
        \left(
          {\vec{p}(\pnt{\xi})}^{T}
          \Dec_{i}( \vec{q}(\pnt{\xi}, t), \pnt{x}(\pnt{\xi}),
          \vec{q}(\pnt{\eta}, t), \pnt{x}(\pnt{\eta}))
        \right)
      \right|}_{\pnt{\eta}=\pnt{\xi}} \refvolm
      \\
    =~&
    \int_{\refelm{K}}
    J^{v}(\pnt{\xi})
    {\vec{p}(\pnt{\xi})}^{T}\vec{g}(\vec{q}(\pnt{\xi}, t),
    \pnt{x}(\pnt{\xi}), t) \refvolm\\
    &
    -
    \sum_{j=1}^{d}
    \int_{\partial \refelm{K}}
    J^{f}(\pnt{\xi})
    {\vec{p}(\pnt{\xi})}^{T}
    \Dec^{*}_{j}(\vec{q}^{-}(\pnt{\xi}, t), \pnt{x}(\pnt{\xi}^{-}),
    \vec{q}^{+}(\pnt{\xi}, t), \pnt{x}(\pnt{\xi}^{+}))
    \;
    n_{j}^{-} \refsurfm,
  \end{split}
\end{equation}
with $G_{jk} = J^{v}\pder{\xi_{j}}{x_{k}}$.

\revb{A quadrature-based DG semi-discretization of~\eqref{eq:ref_var_form} for solution component
$\alpha$ can be written as}
\begin{equation}
  \label{eq:dg:mat}
  \begin{split}
    \mat{M}^{v}
    \der{\vec{q}^{N}_{\alpha}}{t}
    &+
    \mat{V}^{T}
    \sum_{j,k=1}^{d}
    \left(
      \mat{G}_{jk}
      \left(\mat{Q}_{k}\circ \Dec_{j\alpha} \right)
      -
      \left( \Dec_{j\alpha} \circ \mat{Q}_{k}^{T} \right)
      \mat{G}_{jk}
    \right)
    \vec{1}\\
    &=
    {\left(\mat{V}^{v}\right)}^{T} \mat{J}^{v} \mat{W}^{v} \vec{g}^{v}_{\alpha}
    -
    \sum_{j,k=1}^{d}
    {\left(\mat{V}^{f}\right)}^{T} \mat{G}_{jk}^{f} \mat{B}_{k}^{f} \Dec_{j\alpha}^{*}
    .
  \end{split}
\end{equation}
Here $\circ$ denotes the Hadamard (element-wise) product,
the matrix $\Dec_{j\alpha}$ denotes component $\alpha$ of the numerical
flux $\Dec_{j}$ evaluated pairwise at the quadrature nodes with components
\begin{equation}
  \label{eq:flux:matrix}
  \dof{\Dec_{j\alpha}}{n,m}
  =
  \dof{\Dec_{j}\left(\dof{\tilde{\vec{q}}}{n}, \dof{\pnt{x}}{n},
  \dof{\tilde{\vec{q}}}{m}, \dof{\pnt{x}}{m}\right)}{\alpha},
\end{equation}
and the vector $\vec{g}^{v}_{\alpha}$ is the evaluation of the source at the
volume quadrature
\begin{equation}
  \dof{\vec{g}^{v}_{\alpha}}{n}
  =
  \dof{\vec{g}\left(\dof{\tilde{\vec{q}}^{v}}{n}, \dof{\revb{\vec{x}^{v}}}{n}\right)}{\alpha}.
\end{equation}
\revb{
For the element surface quadrature points, the numerical flux vector $\Dec^*_{j\alpha}$ components
are
\begin{equation}
  \dof{\Dec^*_{j\alpha}}{n}
  =
  \dof{\Dec^*_{j}\left(\dof{\tilde{\vec{q}}^{-f}}{n}, \dof{\pnt{x}^{-f}}{n},
  \dof{\tilde{\vec{q}}^{+f}}{n}, \dof{\pnt{x}^{+f}}{n}\right)}{\alpha},
\end{equation}
\revs{where the superscripts $-f$ and $+f$ refer to the inner and outer traces, respectively}.
Since the numeric flux is not symmetric, it will be useful later to define
the numerical flux vectors $\Dec^{-*}_{j\alpha}$ and $\Dec^{+*}_{j\alpha}$ by
their components as
\begin{align}
  \dof{\Dec^{-*}_{j\alpha}}{n} &= \dof{\Dec^*_{j\alpha}}{n},&
  \dof{\Dec^{+*}_{j\alpha}}{n} &=
  \dof{\Dec^*_{j}\left(\dof{\tilde{\vec{q}}^{+f}}{n}, \dof{\pnt{x}^{+f}}{n},
  \dof{\tilde{\vec{q}}^{-f}}{n}, \dof{\pnt{x}^{-f}}{n}\right)}{\alpha}.
\end{align}
}
The vector $\tilde{\vec{q}}$ is defined on the combined volume and surface
quadrature grid from the expansion coefficients $\vec{q}$.
A natural and consistent choice is $\tilde{\vec{q}}_{\alpha} = \mat{V}
\vec{q}^{N}_{\alpha}$, but as will be seen in the next section, for general
element types entropy stability requires that $\tilde{\vec{q}}$ be defined
through an entropy-projection. With the flux matrix
notation~\eqref{eq:flux:matrix} we have the approximation
\begin{equation}
  \begin{split}
    \int_{\refelm{K}}
    G_{jk}(\pnt{\xi})
  &
  {\left.
      \pder{
      }{\xi_{k}}
      \left(
        {p_{\alpha}(\pnt{\eta})}
        \dof{\Dec_{j}(\vec{q}(\pnt{\eta}), \pnt{x}(\pnt{\eta}),
        \vec{q}(\pnt{\xi}), \pnt{x}(\pnt{\xi}))}{\alpha}
      \right)
    \right|}_{\pnt{\eta} = \pnt{\xi}}\refvolm\\
    &\approx
    {\left(\vec{p}_{\alpha}^{N}\right)}^{T}
    \mat{V}^{T}
    \mat{G}_{jk}
    \left(\mat{Q}_{k} \circ \Dec_{j\alpha} \right) \vec{1},\\
  \end{split}
\end{equation}
and
\begin{equation}
  \begin{split}
    \int_{\refelm{K}}
    G_{jk}(\pnt{\xi})
    &
    {\left.
        \pder{
        }{\xi_{k}}
        \left(
          {p_{\alpha}(\pnt{\xi})}
          \dof{
            \Dec_{j}(\vec{q}(\pnt{\xi}), \pnt{x}(\pnt{\xi}),
            \vec{q}(\pnt{\eta}), \pnt{x}(\pnt{\eta}))
          }{\alpha}
        \right)
      \right|}_{\pnt{\eta} = \pnt{\xi}}\refvolm\\
    &\approx
    {\left(\vec{p}_{\alpha}^{N}\right)}^{T}
    \mat{V}^{T}
    \left( \Dec_{j\alpha} \circ \mat{Q}_{k}^{T} \right)
    \mat{G}_{jk}
    \vec{1}.
  \end{split}
\end{equation}

\begin{remark}\label{remark:lgl}
  When the quadrature and interpolation points are both Legendre--Gauss--Lobatto,
  i.e., the DG spectral element method, $\mat{Q}_{i}$ can be replaced with
  $\mat{Q}^{v}_{i}$, the mass matrix $\mat{M}^{v}$ is diagonal, $\mat{P}^{v} =
  \mat{V}^{v} = \mat{I}$, and $\mat{V}^{f}$ selects points from the
  volume quadrature that are on the boundary.
\end{remark}

\subsection{Entropy stable DG with flux differencing in fluctuation form}

By construction, DG scheme~\eqref{eq:dg:mat} is consistent, and here we discuss
what is required to ensure entropy stability.
Following \citet{Castro2013}, we assume the volume numerical flux
satisfies the entropy conservation property\footnote{
  For conservation laws, using the numerical flux in fluctuation
  form~\eqref{eq:num:flux:cons} it is straightforward to show
  that~\eqref{eq:shuffle} reduces to the standard conservative property:
  $\left(\vec{\beta}^{T}(\vec{q}, \vec{x}) - \vec{\beta}^{T}(\vec{p},
  \vec{y})\right) \Fec_{i}(\vec{q}, \vec{x}, \vec{p}, \vec{y})
  =
  \psi_{i}(\vec{q}, \vec{x})
  -
  \psi_{i}(\vec{p}, \vec{y})$ \revb{where $\psi_{i} = \vec{\beta}^{T} \vec{f}_{i} - \zeta_{i}$ is the
  entropy potential}; see \citet{Castro2013}.
},
\begin{subequations}
  \label{eq:shuffle}
  \revb{
  \begin{equation}
    \label{eq:shuffle:volume}
    \vec{\beta}^{T}(\vec{p}, \vec{y})
    \Dec_{i}(\vec{p}, \vec{y}, \vec{q}, \vec{x})
    -
    \vec{\beta}^{T}(\vec{q}, \vec{x})
    \Dec_{i}(\vec{q}, \vec{x}, \vec{p}, \vec{y})
    =
    \zeta_{i}(\vec{q}, \vec{x})
    -
    \zeta_{i}(\vec{p}, \vec{y}),
  \end{equation}
  }
  and the surface flux the entropy dissipation property
  \revb{
  \begin{equation}
    \label{eq:shuffle:surface}
    \sum_{i=1}^{d}
    n_i(x)
    \left(
    \vec{\beta}^{T}(\vec{p}, \vec{y}) \Dec^{*}_{i}(\vec{p}, \vec{y}, \vec{q}, \vec{x})
    -
    \vec{\beta}^{T}(\vec{q}, \vec{x}) \Dec^{*}_{i}(\vec{q}, \vec{x}, \vec{p}, \vec{y})
    \right)
    \le
    \sum_{i=1}^{d}
    n_i(x)
    \left(
    \zeta_{i}(\vec{q}, \vec{x})
    -
    \zeta_{i}(\vec{p}, \vec{y})
    \right);
  \end{equation}
}
\end{subequations}
we refer to relations~\eqref{eq:shuffle:volume} and~\eqref{eq:shuffle:surface}
as conservative and dissipative entropy shuffle relations.
For entropy stability it is critical that the numerical fluxes, source terms,
and entropy variables be evaluated with a $\tilde{\vec{q}}$ that properly
accounts for variational crimes.  To do this we define the auxiliary quantities
\begin{subequations}
  \begin{align}
    \vec{q}_{\alpha}^{v} &= \mat{V}^{v} \vec{q}_{\alpha}^{N},&
    \vec{q}_{\alpha}^{f} &= \mat{V}^{f} \vec{q}_{\alpha}^{N},&
    \vec{q}_{\alpha}     &= \mat{V}     \vec{q}_{\alpha}^{N},&
    \alpha &= 1, \dots, N_{c},\\
    \vec{x}_{i}^{v} &= \mat{V}^{v} \vec{x}_{i}^{N},&
    \vec{x}_{i}^{f} &= \mat{V}^{f} \vec{x}_{i}^{N},&
    \vec{x}_{i}     &= \mat{V}     \vec{x}_{i}^{N},&
    i &= 1, \dots, 3,\\
  \end{align}
\end{subequations}
with the $\vec{q}^{v}$, $\vec{q}^{f}$, $\vec{q}$, $\vec{x}^{v}$,
$\vec{x}^{f}$, and $\vec{x}$ being the solution and coordinate vectors with all
the components.
The expansion coefficients of the entropy variables are defined using the
quadrature $L^{2}$ projection operator~\eqref{eq:quad:proj},
\begin{equation}
  \vec{\beta}^{N}_{\alpha} =
  \mat{P}^{v} \beta_{\alpha}\left( \vec{q}^{v}, \vec{x}^{v} \right),
\end{equation}
with $\vec{\beta}^{N}$ being the vectors of all components; we similarly define
$\vec{\beta}^{v}_{\alpha}$,
$\vec{\beta}^{f}_{\alpha}$,
$\vec{\beta}_{\alpha}$,
$\vec{\beta}^{v}$,
$\vec{\beta}^{f}$, and
$\vec{\beta}$ as above.
With this, the entropy-projected solution vectors are
\begin{subequations}
  \label{eq:qtilde}
  \begin{align}
    \tilde{\vec{q}}^{v}_{\alpha} &= \beta^{-1}_{\alpha}\left( \vec{\beta}^{v}, \vec{x}^{v} \right),\\
    \tilde{\vec{q}}^{f}_{\alpha} &= \beta^{-1}_{\alpha}\left( \vec{\beta}^{f}, \vec{x}^{f}  \right),\\
    \tilde{\vec{q}}_{\alpha} &= \beta^{-1}_{\alpha}\left( \vec{\beta}, \vec{x} \right)
    =
    \begin{bmatrix}
      \tilde{\vec{q}}^{v}_{\alpha}\\
      \tilde{\vec{q}}^{f}_{\alpha}
    \end{bmatrix},
  \end{align}
\end{subequations}
with $\tilde{\vec{q}}^{v}$, $\tilde{\vec{q}}^{f}$, and $\tilde{\vec{q}}$ being
the combined entropy-projected solution vectors.

\begin{lemma}\label{lemma:single:es}
  If numerical flux $\Dec_{j}$ satisfies conservative entropy
  shuffle~\eqref{eq:shuffle} and $\Dec_{j\alpha}$ is evaluated with
  $\tilde{\vec{q}}$ as defined in~\eqref{eq:qtilde} then DG
  scheme~\eqref{eq:dg:mat} on a single element satisfies the entropy
  relationship
  \begin{equation}
    \label{eq:single:es}
    \revb{
      \vec{1}^{T}\mat{J}^{v} \mat{W}^{v} \der{\eta\left(\vec{q}^{v}\right)}{t}\\
    }
      \le
      \vec{1}^{T}
      \mat{J}^{v} \mat{W}^{v}
      \Pi\left( \tilde{\vec{q}}^{v}, \pnt{x}^{v} \right)
      \revb{+}
      \sum_{j,k=1}^{d}
      \left(
        \vec{1}^{T} \mat{G}_{jk}^{+f}\mat{B}_{k}^{+f}\vec{\zeta}_{j}^{+f}
        +
        \sum_{\alpha=1}^{N_{c}}
        {\left(\vec{\beta}^{+f}_{\alpha}\right)}^{T}
        \mat{G}_{jk}^{+f} \mat{B}_{k}^{+f} \Dec_{j\alpha}^{+*}
      \right).
  \end{equation}
\end{lemma}
\begin{proof}
  On a single element, assuming exactness in time, it follows that
  \begin{equation}
    \begin{split}
      \vec{1}^{T}\mat{J}^{v}  \mat{W}^{v} \revb{\der{\eta\left(\vec{q}^{v}\right)}{t}}
    &=
    \sum_{\alpha=1}^{N_{c}}
    {\left(\beta_{\alpha}\left(\vec{q}^{v}\right)\right)}^{T}
    \mat{J}^{v}
    \mat{W}^{v} \der{\vec{q}_{\alpha}^{v}}{t}\\
    &=
    \sum_{\alpha=1}^{N_{c}}
    {\left(\beta_{\alpha}\left(\vec{q}^{v}\right)\right)}^{T}
    {\left(\mat{P}^{v}\right)}^{T}
    {\left(\mat{V}^{v}\right)}^{T}
    \mat{J}^{v}
    \mat{W}^{v} \mat{V}^{v}\der{\vec{q}^{N}_{\alpha}}{t}\\
    &=
    \sum_{\alpha=1}^{N_{c}}
    {\left(\vec{\beta}^{N}_{\alpha}\right)}^{T}
    \mat{M}^{v}
    \der{\vec{q}^{N}_{\alpha}}{t}.
    \end{split}
  \end{equation}
  Using DG scheme~\eqref{eq:dg:mat} then gives
  \begin{equation}
    \begin{split}
      \sum_{\alpha=1}^{N_{c}}
      {\left(\vec{\beta}^{N}_{\alpha}\right)}^{T}
      \mat{M}^{v}
      \der{\vec{q}^{N}_{\alpha}}{t}
      =&
      -
      \sum_{\alpha=1}^{N_{c}}
      \sum_{j,k=1}^{d}
      \vec{\beta}_{\alpha}^{T}
      \left(
        \mat{G}_{jk}
        \left(\mat{Q}_{k}\circ \Dec_{j\alpha} \right)
        -
        \left( \Dec_{j\alpha} \circ \mat{Q}_{k}^{T} \right)
        \mat{G}_{jk}
      \right)
      \vec{1}\\
       &+
       \sum_{\alpha=1}^{N_{c}}
       {\left(\vec{\beta}^{v}_{\alpha}\right)}^{T}
       \mat{J}^{v} \mat{W}^{v} \vec{g}^{v}_{\alpha}
       -
       \sum_{\alpha=1}^{N_{c}}
       \sum_{j,k=1}^{d}
       {\left(\vec{\beta}^{f}_{\alpha}\right)}^{T}
       \mat{G}_{jk}^{f} \mat{B}_{k}^{f} \Dec_{j\alpha}^{*}
       .
    \end{split}
  \end{equation}
  Considering the source term, we have that
  \begin{equation}
    \label{eq:es:source}
    \sum_{\alpha=1}^{N_{c}}
    {\left(\vec{\beta}^{v}_{\alpha}\right)}^{T}
    \mat{J}^{v} \mat{W}^{v} \vec{g}^{v}_{\alpha}
    =
    \sum_{\alpha=1}^{N_{c}}
    {\left(\vec{\beta}^{v}_{\alpha}\right)}^{T}
    \mat{J}^{v} \mat{W}^{v}
    \vec{g}_{\alpha}\left(\tilde{\vec{q}}^{v}, \vec{x}^{v}\right)
    =
    \vec{1}^{T}
    \mat{J}^{v} \mat{W}^{v}
    \Pi\left( \tilde{\vec{q}}^{v}, \pnt{x}^{v} \right).
  \end{equation}
  The volume flux term gives
  \begin{equation}
    \begin{split}
      \sum_{\alpha=1}^{N_{c}}
      \sum_{j,k=1}^{d}
      \vec{\beta}_{\alpha}^{T}
    &
    \left(
      \mat{G}_{jk}
      \left(\mat{Q}_{k}\circ \Dec_{j\alpha} \right)
      -
      \left( \Dec_{j\alpha} \circ \mat{Q}_{k}^{T} \right)
      \mat{G}_{jk}
    \right)
    \vec{1}\\
    &=
    \sum_{\alpha=1}^{N_{c}}
    \sum_{n,m=1}^{N_q}
    \sum_{j,k=1}^{d}
    \dof{\vec{\beta}_{\alpha}}{n}
    \dof{\Dec_{j\alpha}}{nm}
    \left(
      \dof{\mat{G}_{jk}}{n}
      \dof{\mat{Q}_{k}}{nm}
      -
      \dof{\mat{Q}_{k}}{mn}
      \dof{\mat{G}_{jk}}{m}
    \right)\\
    &=
    \sum_{\alpha=1}^{N_{c}}
    \sum_{n,m=1}^{N_q}
    \sum_{j,k=1}^{d}
    \dof{\mat{G}_{jk}}{n}
    \dof{\mat{Q}_{k}}{nm}
    \left(
      \dof{\vec{\beta}_{\alpha}}{n}
      \dof{\Dec_{j\alpha}}{nm}
      -
      \dof{\vec{\beta}_{\alpha}}{m}
      \dof{\Dec_{j\alpha}}{mn}
    \right)\\
    &=
    \sum_{n,m=1}^{N_q}
    \sum_{j,k=1}^{d}
    \dof{\mat{G}_{jk}}{n}
    \dof{\mat{Q}_{k}}{nm}
    \left(
      \dof{\vec{\zeta}_{j}}{m}
      -
      \dof{\vec{\zeta}_{j}}{n}
    \right)\\
    &=
    \sum_{j,k=1}^{d}
    \left(
      \vec{1}^{T}
      \mat{G}_{jk}
      \mat{Q}_{k}
      \vec{\zeta}_{j}
      -
      \vec{\zeta}_{j}^{T}
      \mat{G}_{jk}
      \mat{Q}_{k}
      \vec{1}
    \right)\\
    &=
    \sum_{j,k=1}^{d}
    \vec{1}^{T}
    \mat{G}_{jk}
    \mat{Q}_{k}
    \vec{\zeta}_{j},
    \end{split}
  \end{equation}
  where we have used the entropy conservative shuffle~\eqref{eq:shuffle:volume} and
  exactness of the derivative matrix for constants~\eqref{eq:Qconst}.
  Using SBP property~\eqref{eq:sbp:prop} and geometric conservation law
  assumption~\eqref{eq:gcl} gives
  \begin{equation}
    \label{eq:es:vol}
    \sum_{j,k=1}^{d}
    \vec{1}^{T}
    \mat{G}_{jk}
    \mat{Q}_{k}
    \vec{\zeta}_{j}
    =
    \sum_{j,k=1}^{d}
    \left(
      \vec{1}^{T}
      \mat{G}_{jk}^{f}
      \mat{B}_{k}^{f}
      \vec{\zeta}_{j}^{f}
      -
      \vec{1}^{T}
      \mat{G}_{jk}
      \mat{Q}_{k}^{T}
      \vec{\zeta}_{j}
    \right)
    =
    \sum_{j,k=1}^{d}
    \vec{1}^{T}
    \mat{G}_{jk}^{-f}
    \mat{B}_{k}^{-f}
    \vec{\zeta}_{j}^{-f}.
  \end{equation}
  Now considering the face term yields
  \revb{
  \begin{equation}
    \label{eq:es:face}
    \begin{split}
      -\sum_{\alpha=1}^{N_{c}}
      \sum_{j,k=1}^{d}
    {\left(\vec{\beta}^{f}_{\alpha}\right)}^{T}
    \mat{G}_{jk}^{f} \mat{B}_{k}^{f} \Dec_{j\alpha}^{*}
    &=
    -\sum_{\alpha=1}^{N_{c}}
    \sum_{j,k=1}^{d}
    {\left(\vec{\beta}^{-f}_{\alpha}\right)}^{T}
    \mat{G}_{jk}^{-f} \mat{B}_{k}^{-f} \Dec_{j\alpha}^{-*}\\
    &\le
    -\sum_{j,k=1}^{d}
    \left(
        \vec{1}^{T} \mat{G}_{jk}^{-f}\mat{B}_{k}^{-f}\vec{\zeta}_{j}^{+f}
      - \vec{1}^{T} \mat{G}_{jk}^{-f}\mat{B}_{k}^{-f}\vec{\zeta}_{j}^{-f}
      + \sum_{\alpha=1}^{N_{c}}
    \left(
      {\left(\vec{\beta}^{+f}_{\alpha}\right)}^{T}
      \mat{G}_{jk}^{-f} \mat{B}_{k}^{-f} \Dec_{j\alpha}^{+*}
    \right)
    \right)\\
    &= \phantom{-}
    \sum_{j,k=1}^{d}
    \left(
        \vec{1}^{T} \mat{G}_{jk}^{+f}\mat{B}_{k}^{+f}\vec{\zeta}_{j}^{+f}
      + \vec{1}^{T} \mat{G}_{jk}^{-f}\mat{B}_{k}^{-f}\vec{\zeta}_{j}^{-f}
      + \sum_{\alpha=1}^{N_{c}}
    \left(
      {\left(\vec{\beta}^{+f}_{\alpha}\right)}^{T}
      \mat{G}_{jk}^{+f} \mat{B}_{k}^{+f} \Dec_{j\alpha}^{+*}
    \right)
    \right),
    \end{split}
  \end{equation}
  where we have used
  the dissipative entropy shuffle~\eqref{eq:shuffle:surface}
  and the fact that
  $\sum_{k=1}^{d}\mat{G}_{jk}^{-f}\mat{B}_{k}^{-f} =
  -\sum_{k=1}^{d}\mat{G}_{jk}^{+f}\mat{B}_{k}^{+f}$.
}

  Putting the source~\eqref{eq:es:source}, volume~\eqref{eq:es:vol}, and
  face~\eqref{eq:es:face} contributions together gives the desired result.
  \qed
\end{proof}
\begin{theorem}
  If the conditions of Lemma~\ref{lemma:single:es} are satisfied, then
  DG scheme~\eqref{eq:dg:mat} on a \revb{mesh with periodic boundary conditions} satisfies the entropy
  stability relationship
  \begin{equation}
    \sum_{K\in\mathcal{T}}
    \sum_{\alpha=1}^{N_{c}}
    \vec{1}^{T}\mat{J}^{v,K} \mat{W}^{v,K} \der{\eta\left(\vec{q}^{v,K}\right)}{t}
    \le
    \sum_{K\in\mathcal{T}}
    \vec{1}^{T}
    \mat{J}^{v,K} \mat{W}^{v,K}
    \Pi\left( \tilde{\vec{q}}^{v,K}, \pnt{x}^{v,K} \right),
  \end{equation}
  where $\mathcal{T}$ is the set of all elements.
\end{theorem}
\begin{proof}
  The result follows directly from the entropy
  shuffle~\eqref{eq:shuffle} applied to neighboring face terms of the single
  element entropy relation~\eqref{eq:single:es} of Lemma~\ref{lemma:single:es}.
\end{proof}

\subsection{Entropy conservative flux for the atmospheric Euler equations}

For the atmospheric Euler equations~\eqref{eq:gov:atmos}, we construct
numerical fluxes $\Dec_k$ such that the conservative entropy
shuffle~\eqref{eq:shuffle} is satisfied. Here we use a construction based on the
two-point entropy conserving flux of \citet{Chandrashekar2013} and
\citet{Renac2019}.
First, we define the following auxiliary quantities:
\begin{subequations}
  \begin{align}
    {\left(\rho u_{i}\right)}^{*} &= \avg{\rho}_{\log} \avg{u_{i}},\\
    \label{eq:pstar}
    p^{*} &= \frac{\avg{\rho}}{2\avg{b}},\\
    e^{*} &=
    \frac{1}{2(\gamma - 1)\avg{b}_{\log}}
    + \avg{\phi(\pnt{x})}
    + {\|\avg{\vec{u}}\|}_{2}^{2} - \frac{\avg{{\|\vec{u}\|}_{2}^{2}}}{2},\\
    \label{eq:rhohat}
    \hat{\rho} &= \frac{\avg{b}\avg{\rho}_{\log}}{b^-},
  \end{align}
\end{subequations}
where the average and log-average operators are defined as follows
\begin{subequations}
  \begin{align}
    \avg{a} &= \frac{a^{+} + a^{-}}{2},\\
    \avg{a}_{\log} &= \frac{a^{+} - a^{-}}{\log\left(a^{+}\right) - \log\left(a^{-}\right)}.
  \end{align}
\end{subequations}
To evaluate $\avg{\cdot}_{\log}$ in a numerically stable manner we follow the
approach outlined by \citet[ Appendix B]{Ismail2009}. With these definitions, an
entropy conservative numerical flux for $k = 1,2,3$ is
\begin{subequations}
  \begin{align}
    \Aec_{k}(\vec{q}^{-}, \pnt{x}^{-}, \vec{q}^{+}, \pnt{x}^{+})
    &=
    \frac{1}{2}
    \begin{bmatrix}
      1 & 0 & 0 & 0 & 0 & 0\\
      0 & 1 & 0 & 0 & 0 & \delta_{1k} \hat{\rho}\\
      0 & 0 & 1 & 0 & 0 & \delta_{2k} \hat{\rho}\\
      0 & 0 & 0 & 1 & 0 & \delta_{3k} \hat{\rho}\\
      0 & 0 & 0 & 0 & 1 & 0
    \end{bmatrix},\\
    \Hec_{k}(\vec{q}^{-}, \pnt{x}^{-}, \vec{q}^{+}, \pnt{x}^{+})
    &=
    \begin{bmatrix}
      {\left(\rho u_{k}\right)}^{*}\\
      {\left(\rho u_{k}\right)}^{*} \avg{u_{1}} + \delta_{1k} p^{*}\\
      {\left(\rho u_{k}\right)}^{*} \avg{u_{2}} + \delta_{2k} p^{*}\\
      {\left(\rho u_{k}\right)}^{*} \avg{u_{3}} + \delta_{3k} p^{*}\\
      e^{*} {\left(\rho u_{k}\right)}^{*} + \avg{u_{k}} p^{*}\\
      \phi\left(\vec{x}^{+}\right)
    \end{bmatrix}.
  \end{align}
\end{subequations}
Using these expressions in the definition of the numerical
flux~\eqref{eq:num:fluctuation} gives
\begin{equation}
  \label{eq:num:ec_flux}
  \Dec_{k}\left(\vec{q}^{-}, \pnt{x}^{-}, \vec{q}^{+}, \pnt{x}^{+}\right)
  =
  \begin{bmatrix}
    {\left(\rho u_{k}\right)}^{*}\\
    {\left(\rho u_{k}\right)}^{*} \avg{u_{1}} +
    \delta_{1k} \left( p^{*} + \frac{1}{2} \hat{\rho}\jmp{\phi} \right)\\
    {\left(\rho u_{k}\right)}^{*} \avg{u_{2}} +
    \delta_{2k} \left( p^{*} + \frac{1}{2} \hat{\rho}\jmp{\phi} \right)\\
    {\left(\rho u_{k}\right)}^{*} \avg{u_{3}} +
    \delta_{3k} \left( p^{*} + \frac{1}{2} \hat{\rho}\jmp{\phi} \right)\\
    e^{*} {\left(\rho u_{k}\right)}^{*} + \avg{u_{k}} p^{*}
  \end{bmatrix}
  -
  \begin{bmatrix}
    \rho^{-} u_{k}^{-}\\
    \rho^{-} u_{k}^{-} u_{1}^{-} + \delta_{1k} p^{-}\\
    \rho^{-} u_{k}^{-} u_{2}^{-} + \delta_{2k} p^{-}\\
    \rho^{-} u_{k}^{-} u_{3}^{-} + \delta_{3k} p^{-}\\
    \rho^{-} e^{-} u_{k}^{-} + u_{k}^{-} p^{-}
  \end{bmatrix},
\end{equation}
where $\jmp{\phi} = \phi(\vec{x}^{+}) - \phi(\vec{x}^{-})$.

Since the second term of~\eqref{eq:num:ec_flux} is independent of 
$\vec{q}^{+}$, it need not be included in either the volume or surface numerical
fluxes. \revs{To see this, first notice that for any matrix with identical columns
$\mat{A} = \vec{a} \vec{1}^T$ and arbitrary matrix $\mat{B}$
we have $\left(\mat{A} \circ \mat{B}\right)\vec{1} = \vec{a} \circ \left(\mat{B} \vec{1}\right)$.
Let us denote the part of the numerical flux matrix that is independent of the second argument
by $\Dec_{j\alpha}^0 = \vec{d_{j\alpha}^0} \vec{1}^T$. Using the aforementioned property in
the first derivative term of DG scheme~\eqref{eq:dg:mat} gives
\begin{equation}
    \mat{V}^{T}
    \sum_{j,k=1}^{d}
    \left(
      \mat{G}_{jk}
      \left(\mat{Q}_{k}\circ \Dec^0_{j\alpha} \right)
    \right)
    \vec{1} =
    \mat{V}^{T}
    \sum_{j,k=1}^{d}
      \mat{G}_{jk}
      \vec{d^0}_{j\alpha}\circ \left(\mat{Q}_{k} \vec{1} \right)= \vec{0},
\end{equation}
because $\mat{Q}_{k}$ is exact for constants.
In the derivative term that comes second we use the SBP property of $\mat{Q}_{k}$ to get
\begin{equation}
    -\mat{V}^{T}
    \sum_{j,k=1}^{d}
    \left(
      \left( \Dec^0_{j\alpha} \circ \left(-\mat{Q}_{k} +
    \begin{bmatrix}
      \mat{0} \\
      & \mat{B}_i^{f}
    \end{bmatrix}
     \right)
      \mat{G}_{jk}
     \right)
    \right)
    \vec{1}.
\end{equation}
Using the special form of $\Dec^0_{j\alpha}$ the first term expands to 
$
 \mat{V}^{T}
    \sum_{j,k=1}^{d}
    \left(
      d^0_{j\alpha} \circ \left(\mat{Q}_{k} \mat{G}_{jk} \vec{1} \right)
     \right) = \vec{0},
$
because of geometric conservation law~\eqref{eq:gcl}.
In the second term the boundary matrix selects
$
-\left(\mat{V}^f\right)^{T}
    \sum_{j,k=1}^{d}
    \left( \mat{G}^f_{jk} \mat{B}^f_k \Dec^{0*}_{j\alpha}\right)
$, which cancels with the identical contribution coming from the surface
term on the right hand side of~\eqref{eq:dg:mat}.
}

Entropy-stable fluxes can be constructed by adding entropy dissipation terms
to the entropy conservative flux in a number of ways, the simplest being local Lax-Friedrichs (Rusanov)
dissipation \cite[ Section 6.1]{Ranocha2018Comp}.
In \ref{app:flux} an entropy dissipative flux using a
matrix dissipation term is described
following the methodology presented in
\citet{Winters2017}.
The matrix dissipation term improves upon scalar dissipation by
distinguishing between advective and acoustic waves
using the flux Jacobian eigendecomposition.

\section{Numerical results}
In this section, numerical results verifying correctness and highlighting the
benefits of the novel entropy-stable scheme are presented.
A one-dimensional Sod problem under gravity showcases the flexibility of the method in terms
of different possible choices of quadrature points,
while demonstrating stability in the presence of shocks.
A classical atmospheric benchmark of a thermal bubble convection
is used to numerically validate entropy conservation and highlight
the benefits of the high-order DG scheme with numerical flux dissipation
for a large eddy simulation scale flow with sharp gradients and small-scale features.
A gravity wave test case in a channel, having an exact linear solution,
verifies convergence and the high-order accuracy of the scheme.
Lastly, a baroclinic wave benchmark demonstrates the robustness
of high-order entropy stable numerics
for an archetype of global weather simulation on anisotropic spherical grids.
Unless otherwise noted, the surface flux used in the test problems is the
entropy dissipative flux from \ref{app:flux}; in all cases the entropy
conservative numerical flux~\eqref{eq:num:fluctuation} is used in the volume.
In all examples, the Courant–Friedrichs–Lewy (CFL) number is defined with respect to the minimum
distance between interpolation points and an estimate of the acoustic wave speed based
on the test case initial condition.

\reva{\subsection{Sod shock tube with gravity}}
\citet{Luo2011} introduced a classical Sod shock tube placed in a
gravitational field. The domain is $[0, 1]$, the geopotential $\reva{\phi}=x$, and
the initial condition is defined as follows
\begin{equation}
  (\rho, u, p) =
  \begin{cases}
    (1, 0, 0) & x < 0.5, \\
    (0.125, 0, 0.1) & x \geq 0.5.
  \end{cases}
\end{equation}
\reva{Entropy stable no-flux boundary conditions are used on both sides of the domain~\cite{Svard2021}}.
The domain is resolved using 32 elements with polynomial order 4.
Time integration is performed using the fourth-order five-stage low storage
Runge--Kutta (LSRK54) method from \citet[(5,4) $2N$-Storage RK scheme, solution
$3$]{Carpenter1994} with CFL = 0.2.
To assess the impact of quadrature rule accuracy,
simulations using two sets of quadrature points are considered: $N+1$
Legendre--Gauss--Lobatto (LGL) points and a more accurate $N+2$ point Gauss quadrature.
\begin{figure}
  \centering
  \includegraphics[scale=1]{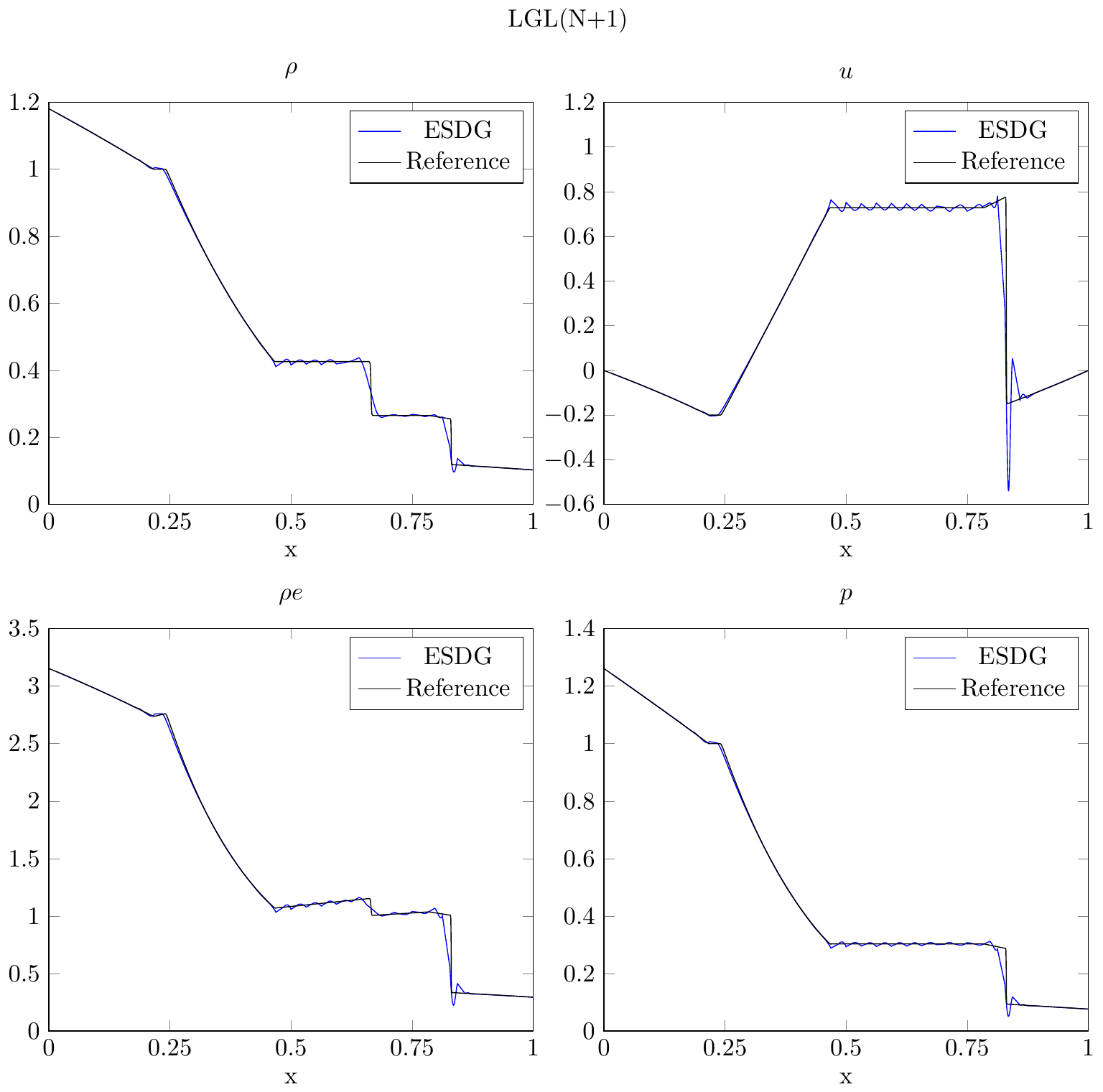}
  \caption{
    Density, velocity, total energy, and pressure at time $t=0.2$
    for the Sod shock tube under gravity. The solution is obtained on a grid
    with 32 elements with polynomial order $N=4$ using an $N+1$ point
    LGL quadrature rule.
    The reference results were obtained using a 5$^{th}$ order WENO scheme with 2000 cells.
  \label{fig:sod_lgl}}
\end{figure}
\begin{figure}
  \centering
  \includegraphics[scale=1]{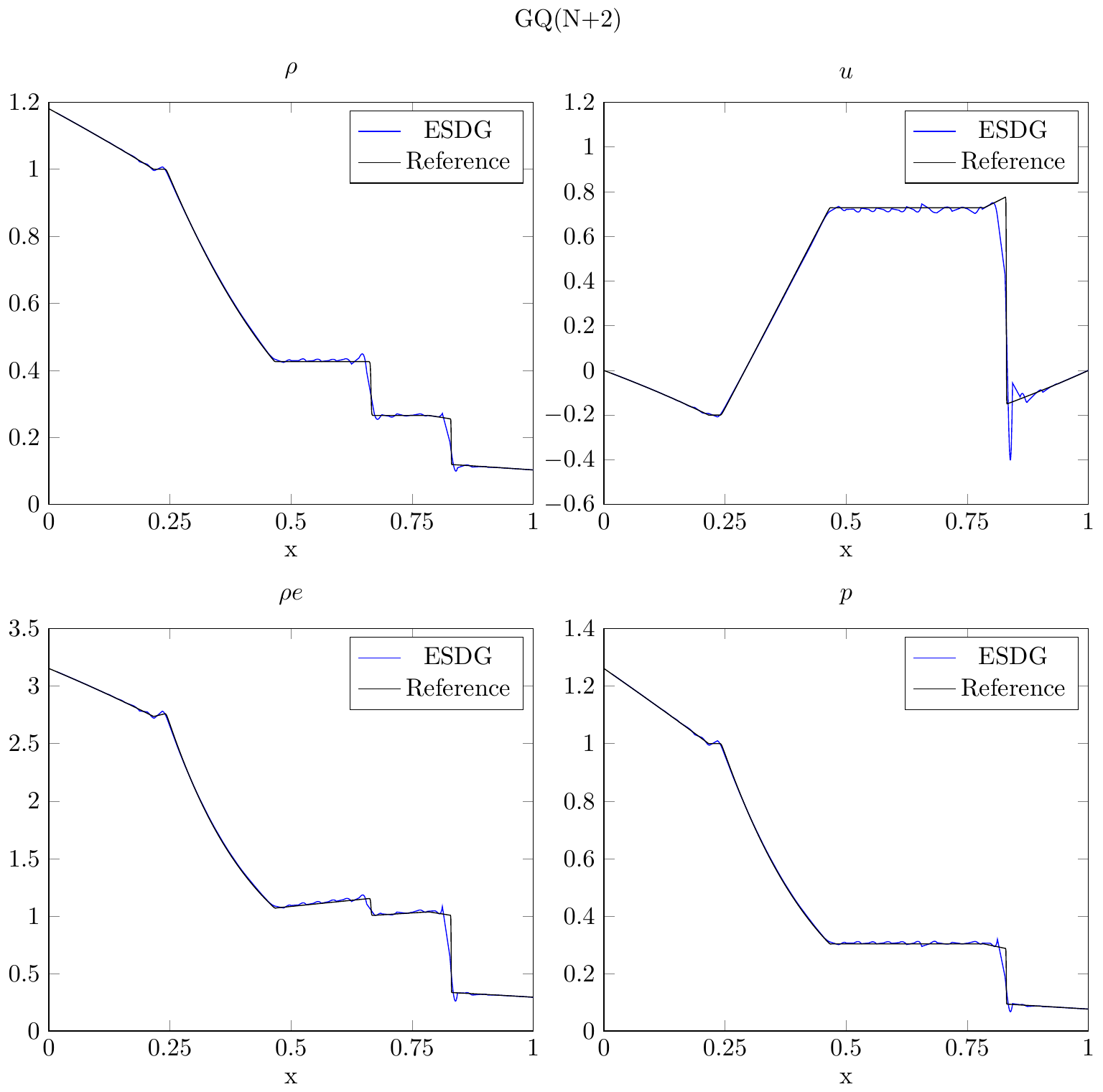}
  \caption{
    Density, velocity, total energy, and pressure at time $t=0.2$
    for the Sod shock tube under gravity. The solution is obtained on a grid
    with 32 elements with polynomial order $N=4$ using an $N+2$ point
    Gauss quadrature rule.
    The reference results were obtained using a 5$^{th}$ order WENO scheme with 2000 cells.
  \label{fig:sod_gll}}
\end{figure}
\autoref{fig:sod_lgl} and \autoref{fig:sod_gll} show the density, velocity,
total energy, and pressure fields at time $t=0.2$ for the LGL and Gauss point
simulations, respectively, along  with a reference simulation that uses a
5$^{th}$ order WENO scheme with 2000 cells.
The results suggest that for this test problem, there is no significant
improvement with more accurate Gauss quadrature.

\reva{\subsection{Rising thermal bubble}}
To verify the entropy conservation and entropy stability of the discretization
we consider the classic atmospheric benchmark of thermal convection.
Our compressible setup closely follows the anelastic experiments presented in
\citet{Smolarkiewicz1992}.
A buoyant bubble is placed in a neutrally stratified environment with
potential temperature 300~K.
The domain $(-\frac{L}{2}, \frac{L}{2}) \times (0, H)$ with $(L, H)$ = (2~km, 2~km) is periodic in the horizontal with rigid-lid boundaries
at the top and bottom.
The initial perturbation is a thermal bubble of radius 250~m with its center
260~m above the bottom and a potential temperature constant value of 0.5~K.
The domain is resolved using $K$ $\times$ $K$ elements with polynomial order $N=4$.
Here, and in the subsequent examples, only the $N+1$ point LGL quadrature rule is considered.
Since the aim of this test is to verify the entropy conservation and stability
properties of the scheme, we augment the LSRK54 scheme with the relaxation
technique from \citet{Ranocha2020}. The CFL number is 0.4.

\autoref{fig:rtb_tht} shows the potential temperature perturbation at time 1000~s computed using
the entropy stable flux on a fine mesh with $K=40$ elements.
Since the scheme has minimal dissipation through the matrix flux, the solution
is rich in small-scale structures a feature that would not be possible with
other stabilization mechanisms at this resolution.
\begin{figure}
  \centering
  \includegraphics[scale=0.4]{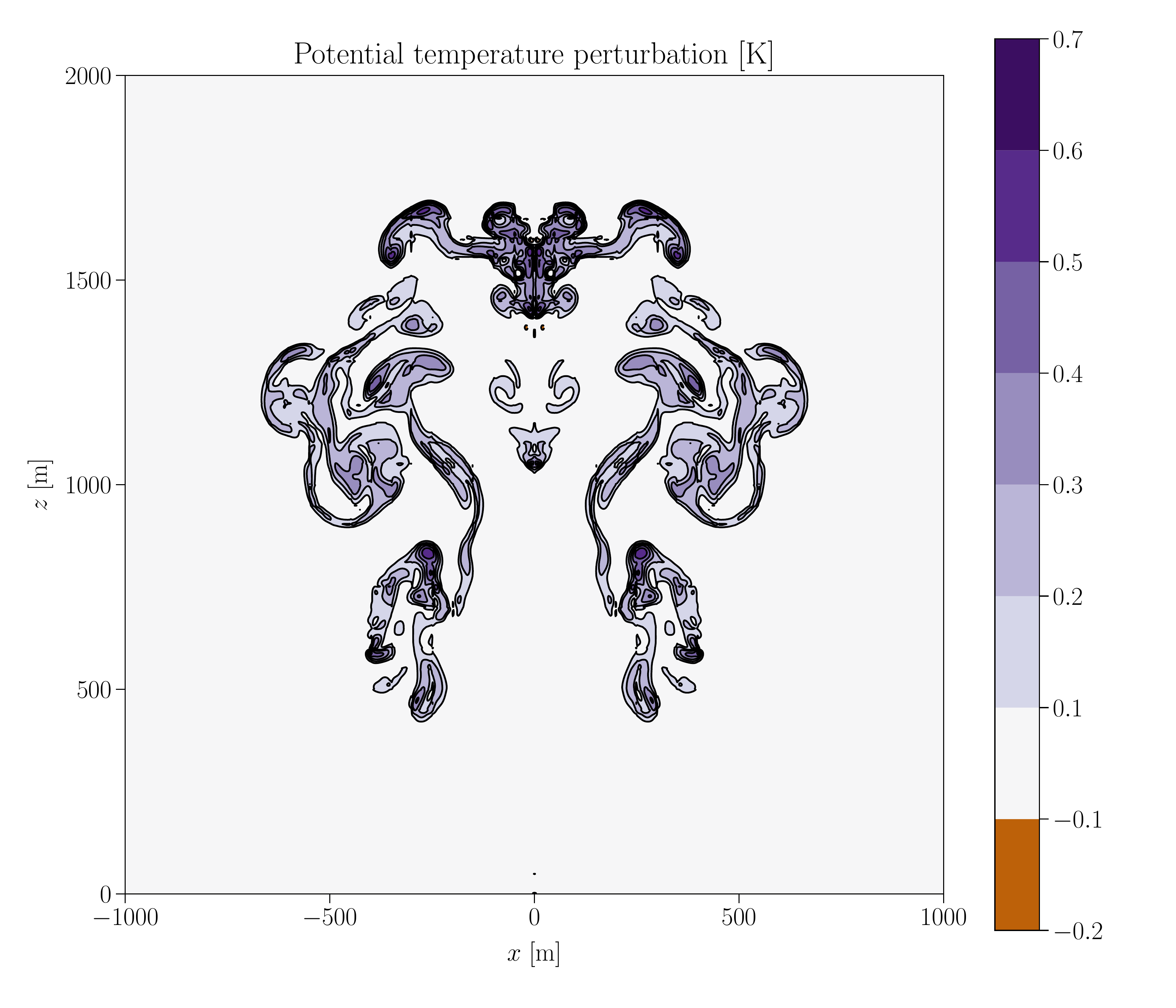}
  \caption{
    Potential temperature perturbation for the rising thermal bubble test case using
    the matrix dissipation flux at time 1000 s on a fine mesh.
  \label{fig:rtb_tht}}
\end{figure}

\autoref{fig:rtb_entropy} shows the normalized entropy change over time for the
rising thermal bubble using both the entropy conservative and entropy
dissipative surface fluxes.
To challenge the scheme, a coarser mesh with $K=10$ elements is used and the
mesh is warped by the transformation
\begin{align}
  \tilde{x}_1 &= x_1 + \frac{L}{5} \sin{\frac{\pi \left(x_1 + \frac{L}{2}\right)}{L}}
                       \sin{\frac{2 \pi x_2}{H}}, \\
  \tilde{x}_2 &= x_2 - \frac{H}{5} \sin{\frac{2\pi \left(x_1 + \frac{L}{2}\right)}{L}}
                       \sin{\frac{\pi x_2}{H}}.
\end{align}
\revb{This mesh warping is illustrated in \autoref{fig:rtb_mesh}.}
\begin{figure}
  \centering
  \includegraphics[scale=0.2]{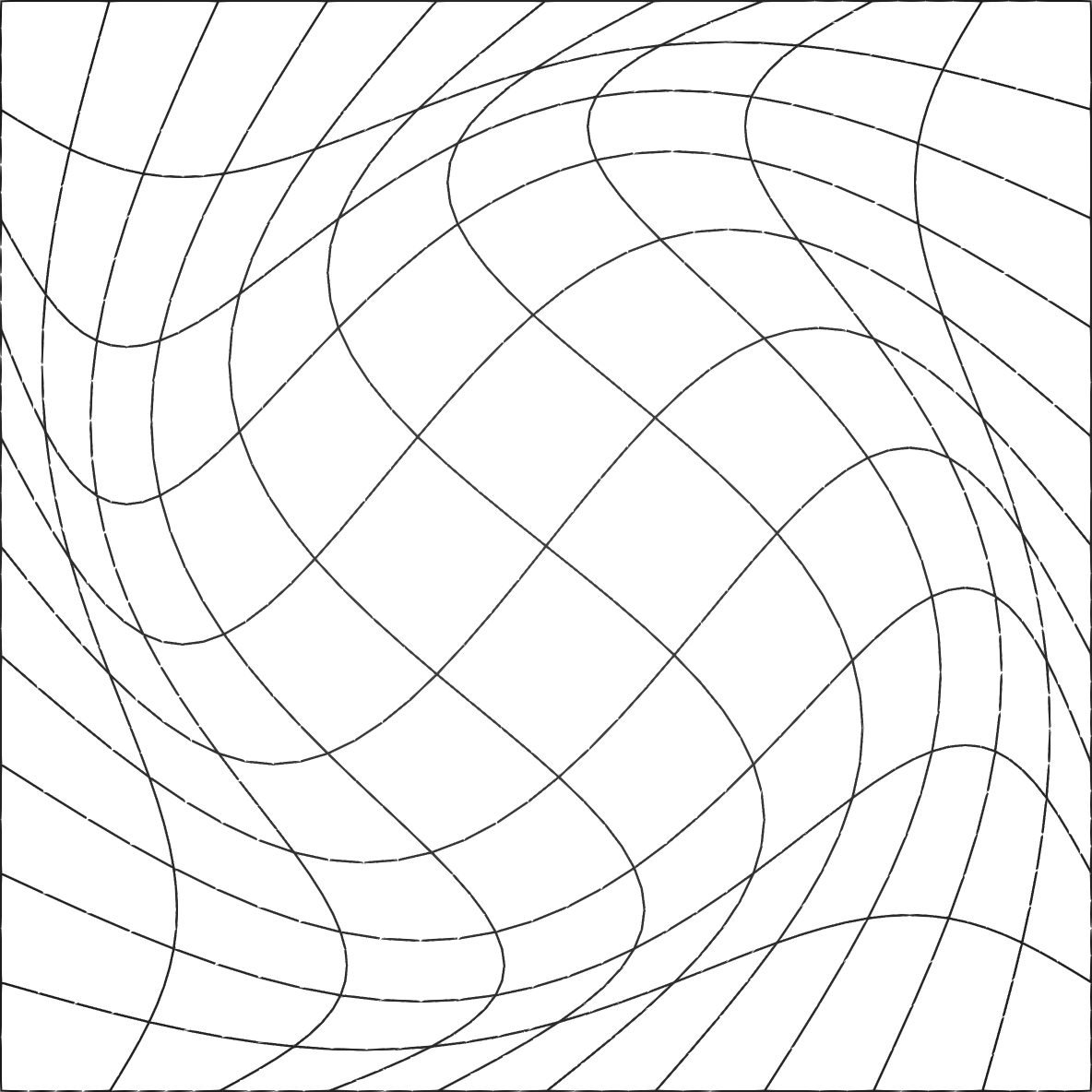}
  \caption{
    Warped curvilinear mesh used in the rising thermal bubble simulations veryfying
    entropy conservation and entropy stability.
  \label{fig:rtb_mesh}}
\end{figure}
\begin{figure}
  \centering
  \includegraphics[scale=1]{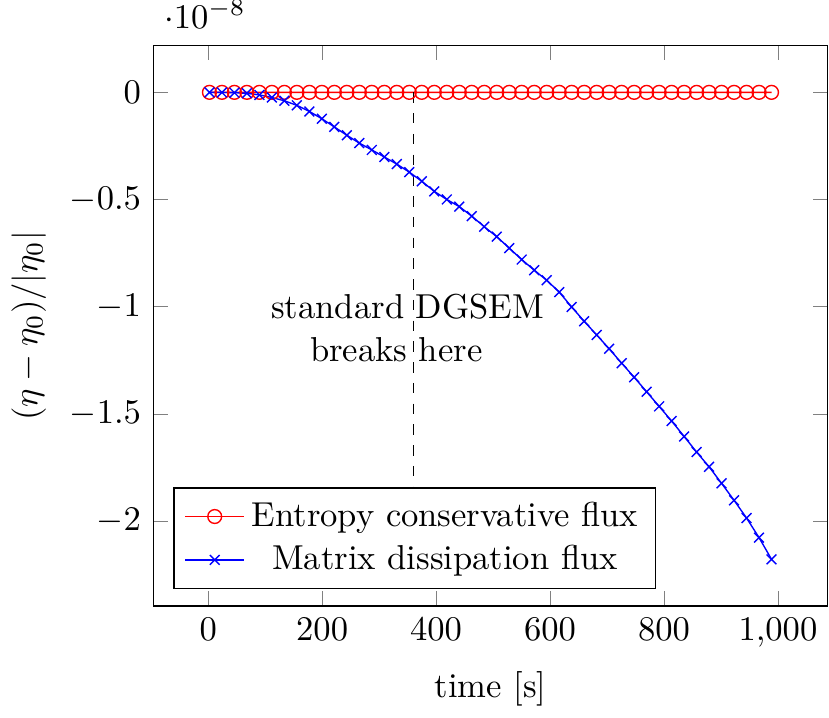}
  \caption{
    Time evolution of normalized entropy change for the rising thermal bubble simulations
    using entropy conservative and entropy stable fluxes on a coarse warped mesh.
    With the entropy conservative flux the maximum value of entropy change is less than
    $2.6 \times 10^{-15}$.
  \label{fig:rtb_entropy}}
\end{figure}
As expected, the simulation using the entropy conservative flux conserves
entropy to machine precision, while the simulation with the entropy stable flux
shows strict entropy decay. Also shown in \autoref{fig:rtb_entropy} is a line
denoting where the standard DG spectral element simulation with a Rusanov flux
and no other stabilization method crashes with a NaN error.

\reva{\subsection{Gravity wave in a channel}}
To verify the high-order convergence of the entropy-stable scheme,
a gravity wave test case from \citet{Baldauf2013} was adopted.
The two-dimensional setup specifies a channel of size $L \times H$
with $L=300$ km and $H=10$ km.
\reva{The boundary conditions are periodic in the horizontal and no-flux in the vertical.
The geopotential is $\phi = g z$ where
$g = 9.81$ m/s${^2}$ and $z$ is the vertical coordinate.}
An isothermal background state with temperature $T_0 = 250$ K
is perturbed by a warm bubble with a maximum temperature perturbation  $\Delta T$
triggering the evolution of gravity and acoustic waves.
A uniform background flow of speed 20 m/s and no Coriolis force is assumed.
An exact solution of the linearized problem is available, which can also be used
to verify the convergence of models solving the full nonlinear equations, provided
the initial perturbation is small enough \cite{Baldauf2013}.
\reva{Since the solution is obtained by complex trigonometric series expansion,
we do not repeat it here, but refer the reader to Section 2 in \citet{Baldauf2013}
for the analytical solution and to Section 3.1 for the parameters of their ``small-scale setup''.}
Experience from \citet{Blaise2016} and \citet{Baldauf2021}, which
used this test case to evaluate high-order DG solvers, shows that
for high orders of accuracy $\Delta T$ has to be very small
to avoid error saturation due to nonlinear effects.
Following \citet{Baldauf2021}, we deviate from the canonical setup in
\citet{Baldauf2013} by decreasing $\Delta T$ to $10^{-3}$~K.

The channel is discretized by a uniform grid with $\revb{K_h} \times \revb{K_v}$ elements.
In all simulations the ratio between the horizontal element size
$\Delta x = \frac{L}{\revb{K_h}}$ and the vertical element size
$\Delta z = \frac{H}{\revb{K_v}}$ is held constant $\frac{\Delta x}{\Delta z} = 3$.
Time integration is done using the LSRK54 time stepper with CFL = 0.1.

\begin{figure}
  \centering
  \includegraphics[scale=0.25]{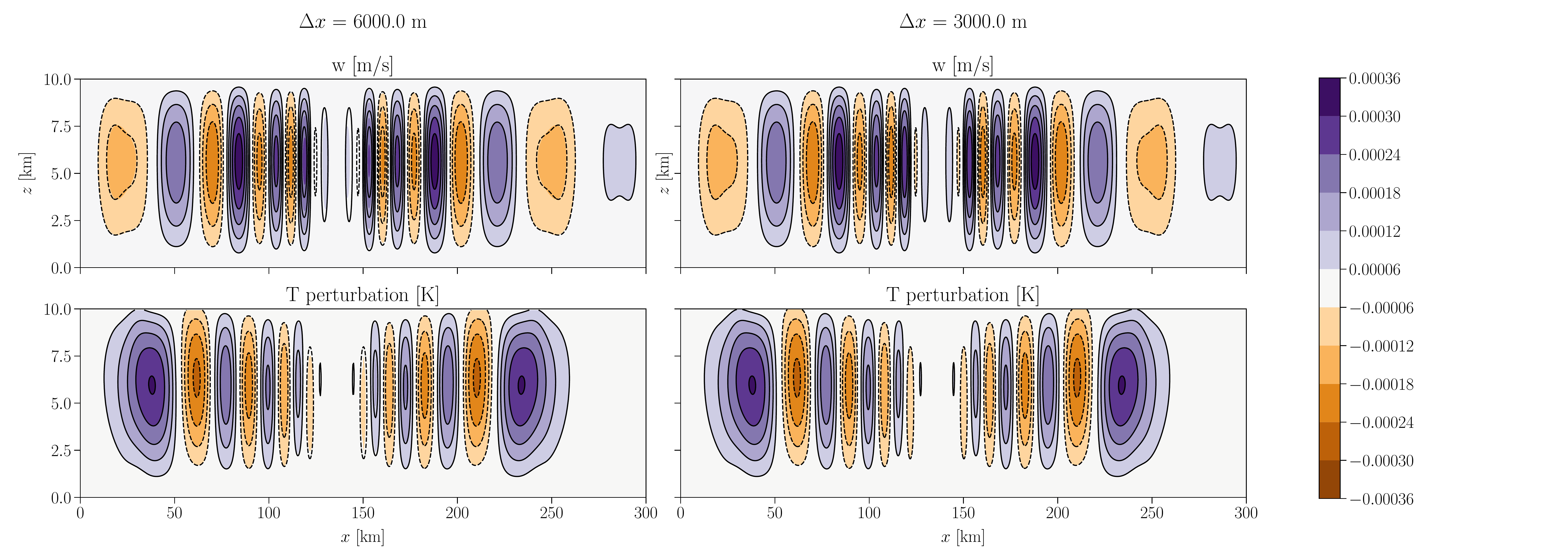}
  \caption{
    Temperature perturbation and vertical velocity for the gravity wave test case
    at time 30 min for grids with horizontal element sizes $\Delta x = 6$ km
    and $\Delta x = 3$ km.
    The contours present the analytical results whereas shading corresponds
    to the numerical solution.
    \label{fig:gw_contours}}
\end{figure}
\autoref{fig:gw_contours} presents temperature and vertical velocity perturbations
at the final simulation time 30 min using polynomial order $N=3$ for coarse and fine grids with $\Delta x = 6$ km and $\Delta x = 3$ km, respectively.
The contours present the analytical results whereas shading corresponds
to the numerical solution.
Small differences between the numerical and analytical solutions
can be observed on the coarse grid, whereas the two are nearly
indistinguishable on the fine grid.
\begin{figure}
  \centering
  \includegraphics[scale=0.8]{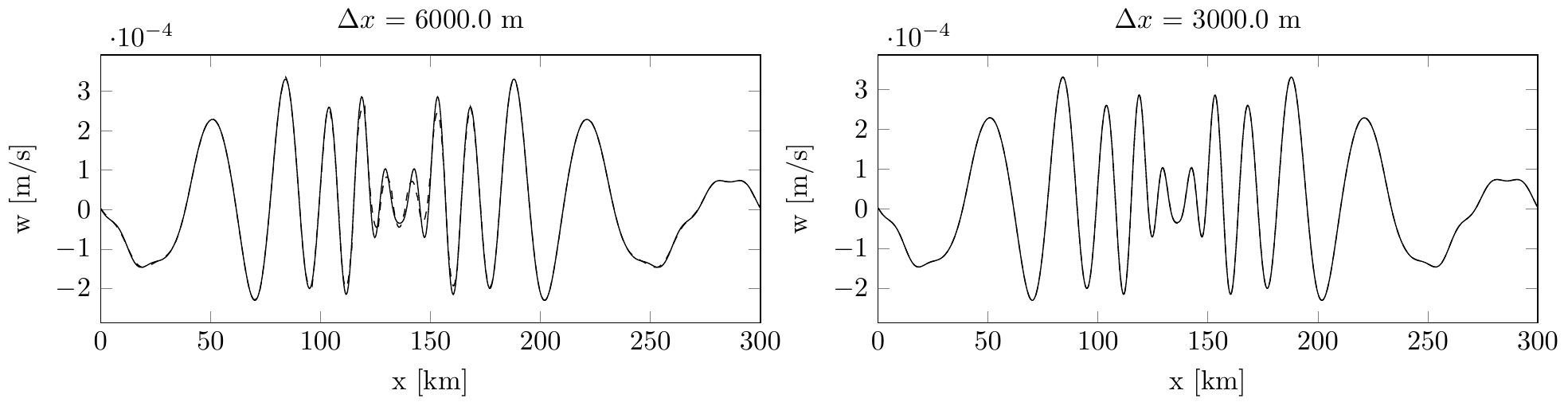}
  \caption{
    Vertical velocity at 5 km height for the gravity wave test case at time 30 min
    for grids with horizontal element sizes $\Delta x = 6$ km and $\Delta x = 3$ km.
    The dashed line shows the simulation results whereas the solid line shows the linear analytic solution.
  \label{fig:gw_line}}
\end{figure}
This is further corroborated by
 \autoref{fig:gw_line} which shows the numerical and analytical
vertical velocity perturbation at $z = 5$ km for the two grids.
The vertical velocity perturbation for $\Delta x = 3$ km
in \autoref{fig:gw_contours} can be directly compared to
Figure 2 in \citet{Baldauf2021}.
With the same polynomial order and resolution
our entropy-stable DG scheme is clearly more accurate, which may
be surprising since \citet{Baldauf2021} used more accurate Gauss
quadrature.
The main reason for this is that, as proven in \ref{app:balance}, the entropy stable scheme
is well-balanced for isothermal atmospheres.
It is well known that well-balanced schemes are better at approximating small
perturbations around balanced states~\cite{Castro2017},
which is exactly the situation in this test problem.
\begin{figure}
  \centering
  \includegraphics[scale=0.9]{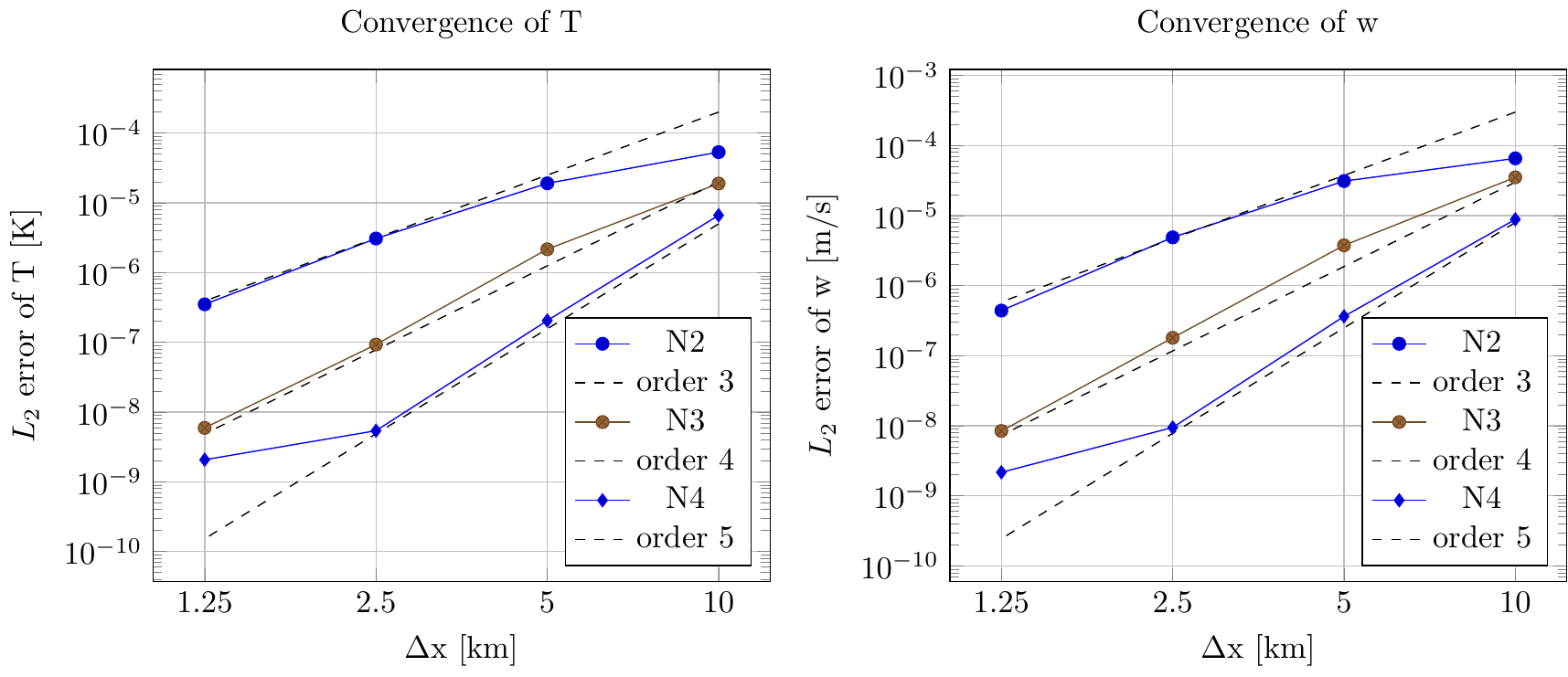}
  \caption{
    Convergence of temperature (left) and vertical velocity (right) errors under
    grid refinement for the gravity wave test case in the $L_2$ norm for
    polynomial orders 2, 3, and 4 \revb{on unwarped meshes}.
    Dashed lines indicate orders 3, 4, and 5
    convergence; the break-down of convergence when the L$^2$ is less than
    $10^{-8}$ is also seen in \citet{Baldauf2021}.\label{fig:gw_conv_l2}}
\end{figure}
\autoref{fig:gw_conv_l2} shows convergence
in the normalized $L_2$ error norm for the temperature and vertical velocity.
The $L_2$ error norm is computed as
\begin{equation}
  L_2(q) = \sqrt{\frac{\int (q - q_{exact}) ^ 2 \volm}{L H}} .
\end{equation}
Overall, we observe convergence rates close to the expected $N+1$ for approximation
with polynomial order $N$. When compared to Figure~3 in \citet{Baldauf2021} we see
much smaller errors of the entropy-stable scheme at coarse resolutions, which is explained
by its well-balanced property. Similarly as \citet{Baldauf2021}, we see
a small influence of nonlinear effects for the highest polynomial order $N=4$,
leading to suboptimal convergence near the finest resolution.
\revb{To verify high-order accuracy of our scheme on curvilinear meshes,
we repeat the convergence experiment on a sequence of meshes warped by the
transformation
\begin{align}
  \tilde{x}_1 &= x_1 + \frac{L}{20} \sin{\frac{\pi \left(x_1\right)}{L}}
                       \sin{\frac{2 \pi x_2}{H}}, \\
  \tilde{x}_2 &= x_2 - \frac{H}{20} \sin{\frac{2\pi \left(x_1\right)}{L}}
                       \sin{\frac{\pi x_2}{H}}.
\end{align}
An example of a mesh warped by this transformation is shown in \autoref{fig:gw_mesh}.
\begin{figure}
  \centering
  \includegraphics[scale=0.3]{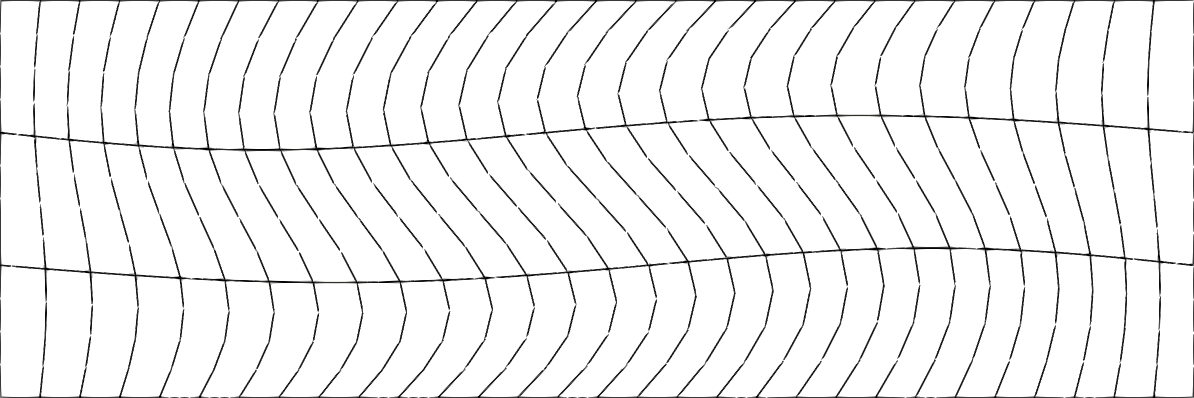}
  \caption{
    An example of warped mesh used to check convergence on curvilinear meshes 
    in the gravity wave experiment. For visualization purposes 
    the plotted mesh is enlarged by a factor of 10 in the vertical direction.
  \label{fig:gw_mesh}}
\end{figure}
\begin{figure}
  \centering
  \includegraphics[scale=0.9]{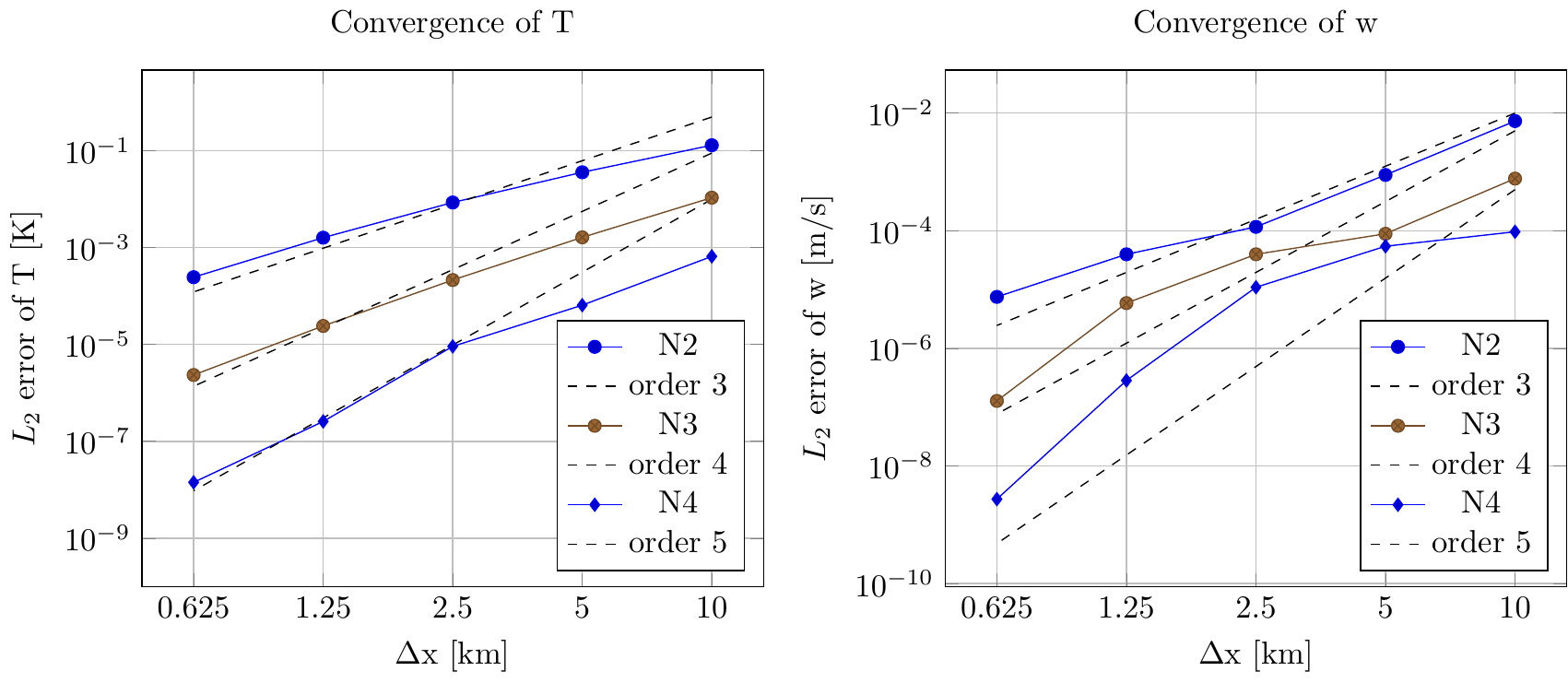}
  \caption{
    Convergence of temperature (left) and vertical velocity (right) errors under
    grid refinement for the gravity wave test case in the $L_2$ norm for
    polynomial orders 2, 3, and 4 on warped meshes. Dashed lines indicate orders 3, 4, and 5
    convergence.
    \label{fig:gw_conv_warp}}
\end{figure}
The convergence results are presented on \autoref{fig:gw_conv_warp}.
Since the scheme is not well-balanced on curved meshes, the magnitude of errors
is much larger on coarse meshes than without warping.
While some of the convergence curves are not
as close to the ideal $N+1$ convergence rate as in the unwarped case, overall 
the scheme achieves increased order of accuracy with increased polynomial order.
The convergence rate for the vertical velocity appears even higher than $N+1$
for polynomial orders 3 and 4.
}

\reva{\subsection{Baroclinic instability}}
Here, we set up a standard benchmark for atmospheric models
in a spherical configuration, used to compare dynamical cores within
the 2016 edition of DCMIP (Dynamical Core Intercomparison Project)~\cite{DCMIP2016},
following the \revo{description in} \citet{Ullrich2013}.
The test case is designed to idealize global weather evolution in the midlatitudes.
The model is initialized with a deep-atmosphere, balanced steady-state
solution that is axisymmetric about the Earth’s rotation axis,
\reva{summarized in appendix A of \citet{Ullrich2013}.}
\reva{The geopotential is $\phi = g r$ where $r$ is the radial distance from
the center of the sphere with $g=9.80616$ m/s$^2$}
and the planetary rotation vector is $\vec{\omega} = (0, 0, 2\Omega)$
with $\Omega = 7.29212 \times 10^{-5} \frac{1}{\text{s}}$ .
A localized Gaussian hill perturbation of the
zonal wind field is introduced in the northern midlatitudes, which triggers
the evolution of a baroclinic wave over the course of several
days. 
\reva{The analytical description of the triggering mechanism can be found 
in Section 6 of \citet{Ullrich2013}}.
In the initial 7-day period, the wave dynamics is linear, after which the
dynamics transitions to a nonlinear regime with a steepening and breaking of the wave.
To compare with other models, in our presentation the focus is on the first
15 days, and particularly day 8 (before the wave breaking) and day 10 (after the wave breaking).
However, to properly evaluate the robustness of the entropy-stable scheme, we confirmed
that all our simulations successfully run for over 100 days without any explicit
dissipation.
To resolve the sphere, we use an equiangular cubed-sphere mesh with $K_h$ elements per horizontal side
on each of the six panels and $K_v$ elements in the vertical.
The model top is 30 km and rigid lid boundary conditions are assumed at the top and
bottom.
We perform two simulations with polynomial orders 3 and 7. The simulation
with polynomial order 3 uses $K_h = 30$ and $K_v = 8$,
while the simulation using polynomial order 7 uses $K_h = 15$ and $K_v = 4$.
These choices ensure that the total number of degree\revo{s} of freedom is the same for both simulations,
corresponding to about 100 km average horizontal resolution and 32 vertical levels.
\revb{For time stepping we use an implicit-explicit formulation, where terms related to the vertical dynamics
were treated implicitly in a linear fashion.
Following \citet{Giraldo2013}, the splitting is achieved by subtracting a linearization of~\eqref{eq:gov:atmos}
from the full equations and then adding it back in; see (3.2) in \citet{Giraldo2013}.
The linearized equations are discretized using a standard (non entropy stable) DG
method based on the same polynomial order and quadrature rule as the entropy stable discretization.
To integrate the split equations in time, the second order additive Runga Kutta scheme
from \citet{Giraldo2013} was used with CFL = 3.0.
}
\begin{figure}
  \centering
  \includegraphics[scale=0.3, trim=250 0 250 0]{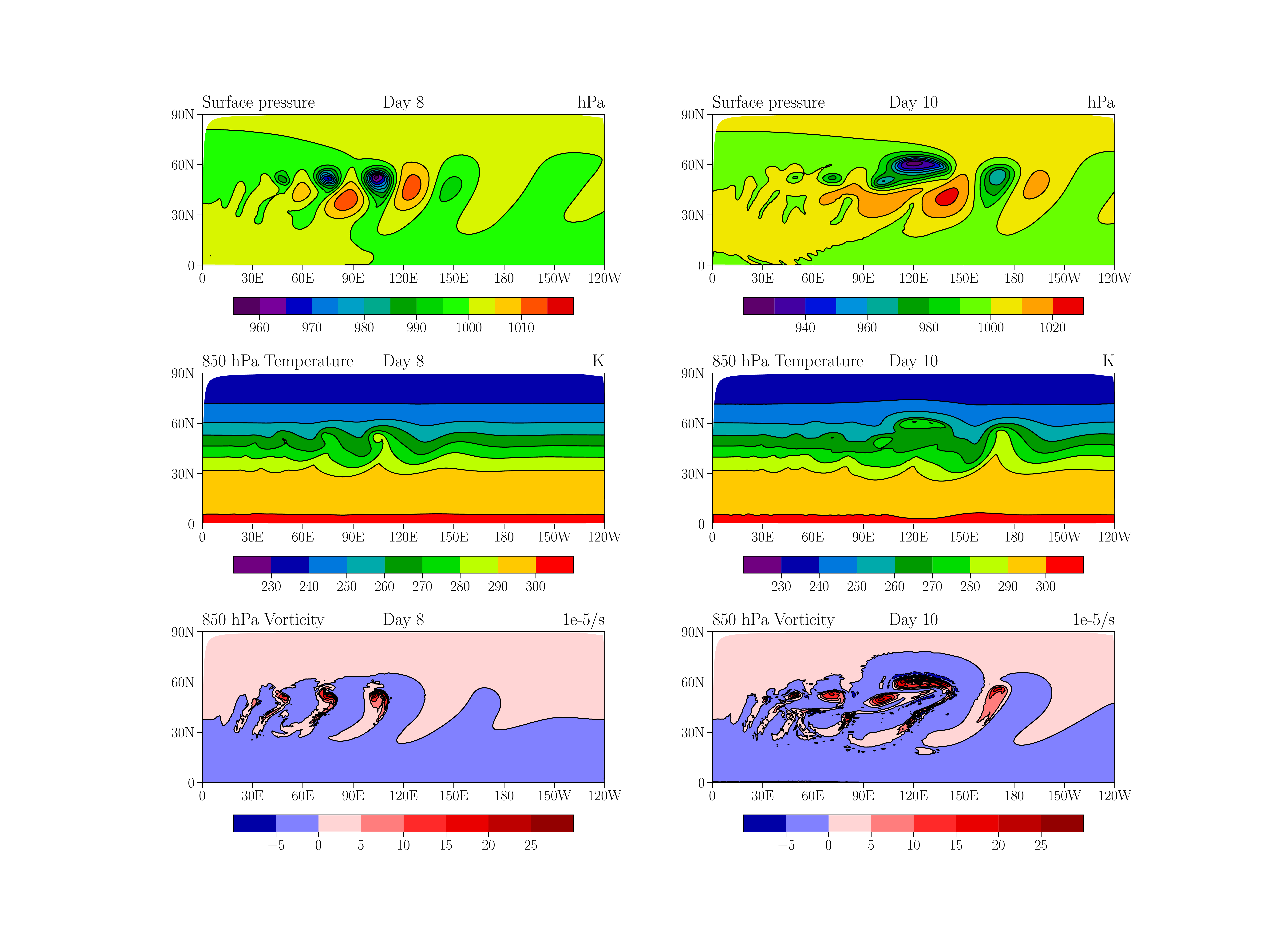}
  \caption{
    Results for the baroclinic wave benchmark using polynomial order 3
    at day 8 (left) and 10 (right).
    Top to bottom: surface pressure, temperature at 850 hPa, and relative vorticty at 850 hPa.
  \label{fig:bw_panel3}}
\end{figure}
\begin{figure}
  \centering
  \includegraphics[scale=0.3, trim=250 0 250 0]{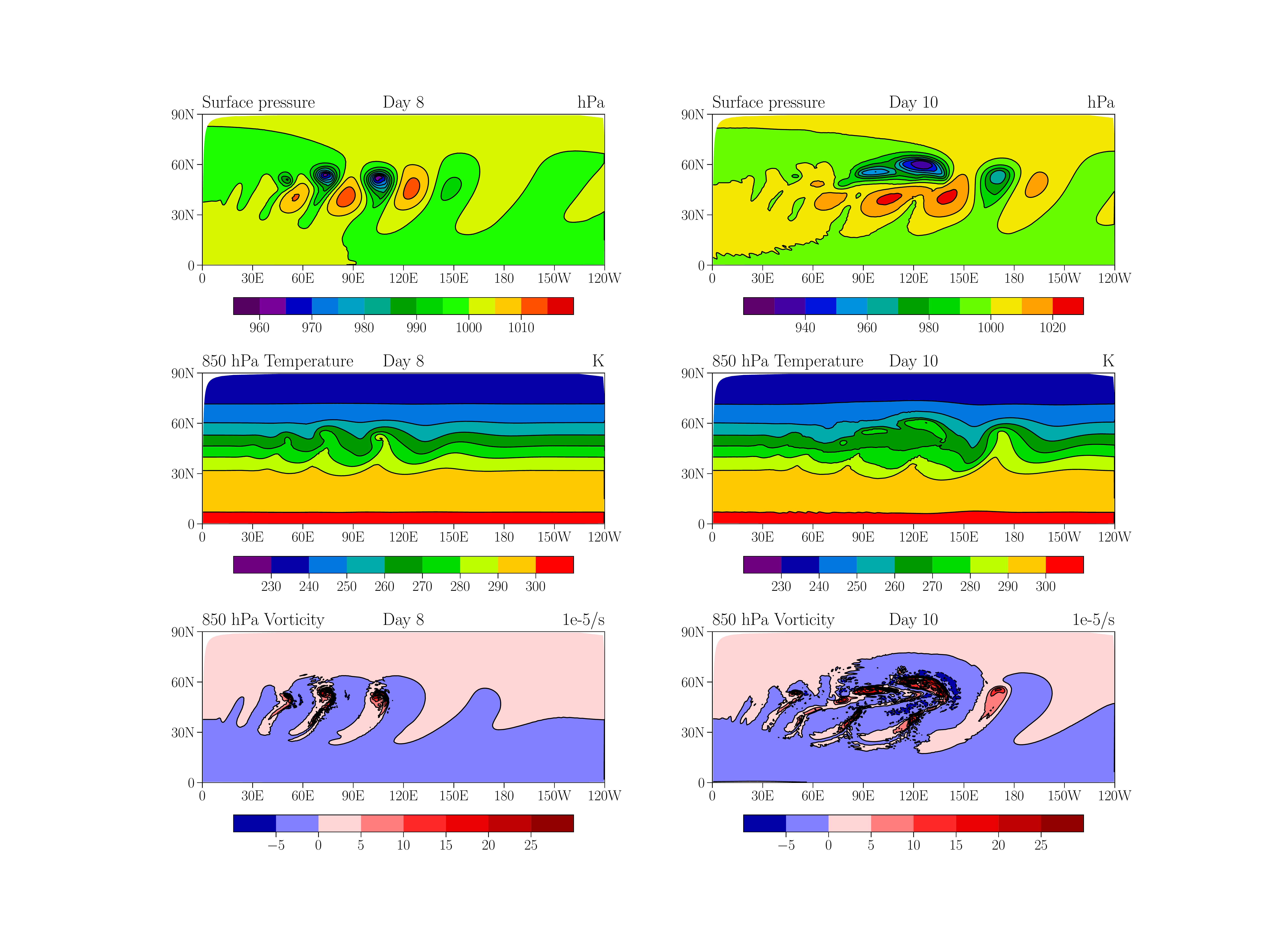}
  \caption{
    Results for the baroclinic wave benchmark using polynomial order 7
    at day 8 (left) and 10 (right).
    Top to bottom: surface pressure, temperature at 850 hPa, and relative vorticty at 850 hPa.
    \label{fig:bw_panel7}}
\end{figure}
\autoref{fig:bw_panel3} and \autoref{fig:bw_panel7} show plots of surface pressure,
850 hPa temperature and relative vorticity at days 8 and 10 for polynomial orders 3 and 7, respectively.
These results can be directly compared to results from the MCore and ENDGame
models~\cite[ Figures 4 and 5]{Ullrich2013} and
the MPAS model~\cite[ Figure 2]{Skamarock2021}. The pressure and temperature fields in their overall
structure look smooth and compare very well with other models, showing only small amounts of grid-scale
oscillations due to minimal numerical diffusion of our method.
More noise can be seen in the relative vorticity fields, especially
with polynomial order 7. This is consistent with \citet{Skamarock2021},
which found that the relative vorticity field is more sensitive to
numerical dissipation than the pressure and temperature fields.
\begin{figure}
  \centering
  \includegraphics[scale=0.9]{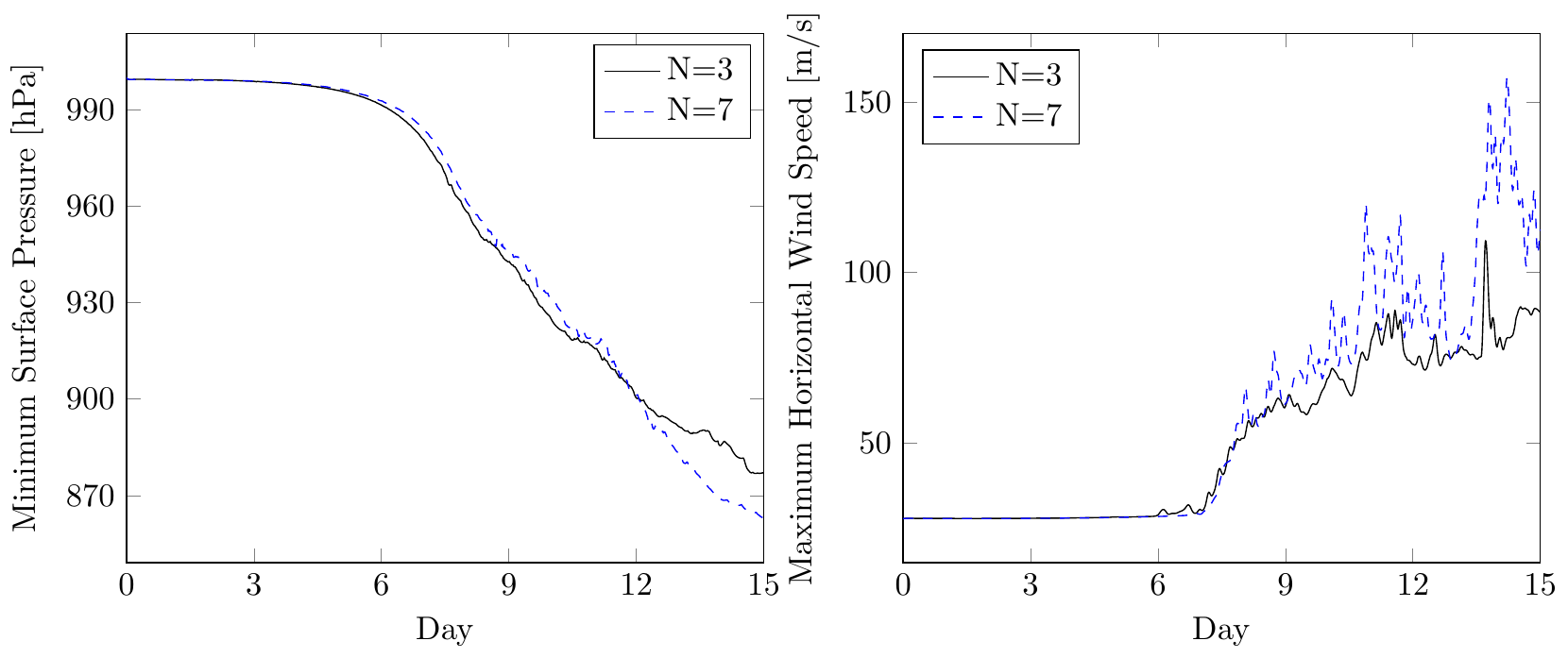}
  \caption{Time evolution of the minimum surface pressure (left) and maximum horizontal
          wind speed (right) for the baroclinic wave benchmark using polynomial
          orders 3 and 7.
    \label{fig:bw_tseries}}
\end{figure}

\autoref{fig:bw_tseries} shows the time evolution of the minimum surface
pressure and maximum horizontal velocity for both polynomial orders,
and can be compared to results from the MPAS model in \citet[ Figures 3 and
4]{Skamarock2021}.
The surface pressure evolution is close for both polynomial orders up until day 12 and compares well
to other models. The maximum horizontal speed shows consistently higher maxima
for polynomial order 7.
This further confirms that the higher order simulation introduces less numerical dissipation.
The presented results successfully demonstrate
that the new entropy-stable scheme is robust and provides high-quality solutions
using only numerical flux dissipation for very high orders of accuracy in
atmospheric simulations.
For realistic applications, e.g., climate simulations and weather prediction, it might be beneficial to add more dissipation to damp grid-scale
oscillations; an advantage of the entropy-stable scheme is that any additional regularization can be guided solely by the needs of the model physics since the dynamics is provably stable.

\section{Concluding Remarks}

We have extended the work of \citet{Renac2019} to develop a flux differencing
formulation of the DG method for non-conservative balance laws on elements that
have operators that can be skew-hybridized \cite{Chan2019}.
By specifying the volume and surface numerical fluxes so that they satisfy an
additional entropy conservation and dissipation relation, the scheme was proved to
be entropy stable.
Based on the new framework, an entropy-stable DG discretization
of the Euler equations with gravity has been constructed.
Entropy-conservative and entropy-stable numerical fluxes for the
Euler equations with gravity have been given in explicit form.
The entropy stability and high-order accuracy
of the novel scheme have been verified on standard atmospheric benchmark problems.
As the atmospheric test problems show, the scheme is quite robust even without
any explicit stabilization techniques beyond the choice of numerical fluxes.

\section*{Acknowledgements}
Waruszewski, Kozdon, Wilcox, and Giraldo were supported by the generosity of
Eric and Wendy Schmidt by recommendation of the Schmidt Futures program.
Waruszewski, Kozdon, Wilcox, Gibson, and Giraldo were supported by the
National Science Foundation under grant AGS-1835881. Giraldo was also supported
by Office of Naval Research under grant N0001419WX00721.

\revo{
\section*{Data availability statement}
The code to generate, analyze, and plot the datasets used in this study can
be found in the associated reproducibility repository
\url{https://github.com/mwarusz/paper-esdg-nonconservative-gravity}.}

\appendix

\section{Entropy stable flux with matrix dissipation}\label{app:flux}
To construct an entropy stable flux the normal entropy conservative flux
$\Dec_n = n_1 \Dec_1 + n_2 \Dec_2 + n_3 \Dec_3$ is combined with a
dissipation operator $\mat{H}_n$ acting on the jump in entropy variables,
leading to the normal entropy stable flux
\begin{equation}
  \Dec^*_n = \Dec_n - \frac{1}{2} \mat{H}_n \jmp{\vec{\beta}}.
\end{equation}
To ensure entropy dissipation and stability, the operator $\mat{H}_n$ has to be constructed carefully~\cite{Derigs2017}.
Here, motivated by \citet{Winters2017}, the dissipation operator is constructed in the so-called
``matrix'' form
\begin{equation}
  \mat{H}_n = \mat{R}_n^T | \mat{\Lambda}_n| \mat{T}_n \mat{R}_n,
\end{equation}
where $\mat{R}_n$ is the matrix of right eigenvectors of the normal flux Jacobian,
$|\mat{\Lambda}_n|$ is the corresponding
matrix of absolute eigenvalues, and $T_n$ is a scaling matrix.
These matrices are given explicitly as
\begin{equation}
  \mat{R}_{n}
  =
  \begin{bmatrix}
    1                  & 1                                     & 0                             & 0   & 1\\
    \avg{u_1} - c^* n_1& \avg{u_1}                             & m_1                        & \ell_1 & \avg{u_1} + c^* n_1\\
    \avg{u_2} - c^* n_2& \avg{u_2}                             & m_2                        & \ell_2 & \avg{u_2} + c^* n_2\\
    \avg{u_3} - c^* n_3& \avg{u_3}                             & m_3                        & \ell_3 & \avg{u_3} + c^* n_3&\\
    h^* - u_n c^*      & \frac{\overline{u^2}}{2} + \avg{\phi} & \vec{m} \cdot \avg{\vec{u}}& \vec{\ell} \cdot \avg{\vec{u}} & h^* + u_n c^*
  \end{bmatrix},
\end{equation}
\begin{equation}
  |\mat{\Lambda}|_{n} =
    \Diagonal\left(
    |u_n - c^*|,
    |u_n|,
    |u_n|,
    |u_n|,
    |u_n + c^*|
    \right),
\end{equation}
\begin{equation}
  \mat{T}_{n} =
    \Diagonal\left(
      \frac{\avg{\rho}_{\log}}{2 \gamma},
      \frac{(\gamma - 1) \avg{\rho}_{\log}}{\gamma},
      p^*,
      p^*,
      \frac{\avg{\rho}_{\log}}{2 \gamma}
      \right),
\end{equation}
where
\begin{align}
  u_n &= \vec{n} \cdot \avg{\vec{u}}, \\
  \overline{u^2} &= 2 {\|\avg{\vec{u}}\|}_{2}^{2} - \avg{{\|\vec{u}\|}_{2}^{2}}, \\
  p^* &= \frac{\avg{\rho}}{2 \avg{b}}, \\
  c^* &= \sqrt{\frac{p^*}{\avg{\rho}_{\log}}}, \\
  h^* &= \frac{\gamma}{2 \avg{b}_{\log} (\gamma - 1)} + \overline{\frac{u^2}{2}} + \avg{\phi}, \\
\end{align}
and $\vec{n}$, $\vec{m}$, $\vec{\ell}$ are orthonormal vectors.
The action of the dissipation
operator on the jump in entropy variables can be written without having
to specify non-uniquely defined tangent vectors $\vec{m}$ and $\vec{\ell}$ as
\begin{align}
  \left(\mat{H}_n \jmp{\vec{\beta}}\right)_\rho &=
    w_1 + w_2 + w_3, \\
  \begin{split}
  \left(\mat{H}_n \jmp{\vec{\beta}}\right)_{\vec{\rho u}} &=
    \left(w_1 (\avg{\vec{u}} - c^* \vec{n})  +
          w_2 \avg{\vec{u}} +
          w_3 (\avg{\vec{u}} + c^* \vec{n}) \right. \\
          &\qquad \left. +|u_n| p^* \mat{T} \left(\jmp{\vec{\beta}}_{\vec{\rho u}} -
                                              \avg{\vec{u}} \jmp{\vec{\beta}}_{\rho e} \right) \right),
  \end{split} \\
  \begin{split}
  \left(\mat{H}_n \jmp{\vec{\beta}}\right)_{\rho e} &=
  \left(w_1 (h^* - c^* u_n)  +
        w_2 \left(\overline{\frac{u^2}{2}} + \avg{\phi} \right) +
        w_3 (h^* + c^* u_n) \right. \nonumber \\
        &\qquad \left. + |u_n| p^* \left(\mat{T} \avg{\vec{u}} \cdot \jmp{\vec{\beta}}_{\vec{\rho u}} -
                        \| \mat{T} \avg{\vec{u}} \|_2^2 \jmp{\vec{\beta}}_{\rho e}
                  \right)
              \right),
  \end{split}
\end{align}
where
\begin{align}
  \mat{T} &= \mat{I} - \vec{n} \otimes \vec{n}, \\
  w_1 &= |u_n - c^*| \frac{\avg{\rho}_{\log}}{2 \gamma} \left(\jmp{\vec{\beta}}_\rho +
                    (\jmp{\vec{u}} - c^* \vec{n}) \cdot \jmp{\vec{\beta}}_{\vec{\rho u}}
                    + (h^* - c^* u_n) \jmp{\vec{\beta}}_{\rho e}\right), \\
  w_2 &= |u_n| \frac{(\gamma - 1) \avg{\rho}_{\log}}{\gamma}  \left(\jmp{\vec{\beta}}_\rho +
                    \avg{\vec{u}} \cdot \jmp{\vec{\beta}}_{\vec{\rho u}}
                    + \left(\overline{\frac{u^2}{2}}
                    + \avg{\phi} \right) \jmp{\vec{\beta}}_{\rho e}\right), \\
  w_3 &= |u_n + c^*| \frac{\avg{\rho}_{\log}}{2 \gamma}  \left(\jmp{\vec{\beta}}_\rho +
                    (\jmp{\vec{u}} + c^* \vec{n}) \cdot \jmp{\vec{\beta}}_{\vec{\rho u}}
                    + (h^* + c^* u_n) \jmp{\vec{\beta}}_{\rho e}\right),
\end{align}
and $\jmp{\vec{\beta}}_{\rho}$, $\jmp{\vec{\beta}}_{\vec{\rho u}}$,
$\jmp{\vec{\beta}}_{\rho e}$ denotes $\jmp{\vec{\beta}}_1$,
$(\jmp{\vec{\beta}}_2, \jmp{\vec{\beta}}_3,\jmp{\vec{\beta}}_4)$, $\jmp{\vec{\beta}}_5$,
respectively.

\section{Well-balanced property for isothermal atmosphere}\label{app:balance}
Here we prove that the new entropy-stable scheme is well-balanced for an
isothermal atmosphere for an unwarped ``box'' domain.
That is we assume that the domain is $\Omega$ is such that all the elements are
aligned with $\xi_{3}$ being the direction that gravity acts and $\xi_{1}$ and
$\xi_{2}$ are orthogonal. This is possible on, for instance, a perfect spherical
shell domain with gravity acting radially or an unwarped hexahedral domain with
gravity acting downward. In both cases we let $r$ be the direction that gravity
acts so that the geopotential is $\phi = g r$ where $g$ is a constant.
In each element we assume that $r \equiv r(\xi_{3})$ so that $\phi$ is also only
a function of $\xi_{3}$.
To demonstrate hydrostatic balance we show that the entropy conservative
flux~\eqref{eq:num:ec_flux} is zero when the initial conditions are set to the
analytic hydrostatic state for an isothermal atmosphere.

A hydrostatically balanced state is a state which is independent of time with a
zero velocity and depends only on $r$. After assuming a stationary solution
which only depends on $r$ with $u_{k} = 0$, the atmospheric
equations~\eqref{eq:gov:atmos} reduce to the hydrostatic balance equation
\begin{equation}
  \pder{p}{r} = -\rho g.
\end{equation}
Under the isothermal assumption of the ideal gas law $p = \rho R_d T_0$ with $R_{d}$
and $T_{0}$ being the gas constant and constant reference temperature,
respectively, we then have an exponential solution for pressure $p$ and density
$\rho$
\begin{subequations}
  \label{eq:ap:balanced}
  \begin{align}
    p &= p_0 \exp{\left(-\frac{g r}{R_d T_0}\right)}, \label{eq:ap1:p} \\
    \rho &= \frac{p_0}{R_d T_0} \exp{\left(-\frac{g r}{R_d T_0}\right)}, \label{eq:ap1:rho}
  \end{align}
\end{subequations}
where $p_0$ is the pressure at $r = 0$.

Since the solution only varies in $r$ and thus is constant with respect to
$\xi_{1}$ and $\xi_{2}$ in each element, consistency of a numerical flux in
fluctuation form~\eqref{eq:num:fluctuation} implies that $\Dec_{1}$ and
$\Dec_{2}$ are zero.
The third component of the entropy conservative numerical
flux~\eqref{eq:num:ec_flux} simplifies to
\begin{equation}
  \label{eq:ap1:dec3}
  \Dec_{3}\left(\vec{q}^{-}, \pnt{x}^{-}, \vec{q}^{+}, \pnt{x}^{+}\right)
  =
  \begin{bmatrix}
    0\\
    0\\
    0\\
    \left( p^{*} + \frac{1}{2} \hat{\rho}\jmp{\phi} \right)\\
    0
  \end{bmatrix}
  -
  \begin{bmatrix}
    0\\
    0\\
    0\\
    p^{-}\\
    0
  \end{bmatrix}.
\end{equation}
Since in the balance solution~\eqref{eq:ap:balanced} pressure is proportional to
density the inverse temperature is constant, $b = \rho / 2p = 1 / R_d T_0$,
and the expressions for the auxiliary variables~\eqref{eq:pstar}
and~\eqref{eq:rhohat} simplify to
\begin{subequations}
  \begin{align}
    p^{*} &= \frac{R_d T_0}{2} \avg{\rho}, \label{eq:ap1:pstar}\\
    \hat{\rho} &= \avg{\rho}_{\log}. \label{eq:ap1:rhohat}
  \end{align}
\end{subequations}
Evaluating the density logarithmic average with the balanced density
solution~\eqref{eq:ap1:rho} gives
\begin{equation}
  \label{eq:ap1:rholog}
  \hat{\rho}
  =
  \frac{\rho^+ - \rho^-}{\log{\rho^+} - \log{\rho^-}} =
  -R_d T_0 \frac{\rho^+ - \rho^-}{g r^+ - g r^-} =
  -R_d T_0 \frac{\jmp{\rho}}{\jmp{\phi}}.
\end{equation}
Using the auxiliary expressions~\eqref{eq:ap1:pstar} and~\eqref{eq:ap1:rholog}
in the non-zero term of the numerical flux~\eqref{eq:ap1:dec3} leads to
\begin{equation}
  \left( p^{*} + \frac{1}{2} \hat{\rho}\jmp{\phi} \right) - p^- =
  \frac{R_d T_0}{2} \avg{\rho} - \frac{R_d T_0}{2} \jmp{\rho} - p^- =
  R_d T_0 \rho^- - p^- = p^- - p^- = 0.
\end{equation}
This implies that $\Dec_{3}$ is zero and thus the solution is well-balanced.

\bibliographystyle{spmpscinat}
\bibliography{bibliography/blesdg, bibliography/Giraldo_refs}

\end{document}